\documentclass[12pt]{amsart}

\usepackage{color}
\usepackage{mathrsfs}
\usepackage{amssymb}
\usepackage{graphicx}  

\def\pasdegrille{\let\grille = \pasgrille}
\def\ecriture#1#2{\setbox1=\hbox{#1}
\dimen1= \wd1
\dimen2=\ht1
\dimen3=\dp1
\grille #2 \box1 }
\def\aat#1#2#3{
\divide \dimen1 by 48
\dimen3=\dimen1
\multiply \dimen1 by #1
\advance \dimen1 by -\dimen3
\divide \dimen1 by 101
\multiply \dimen1 by 100
\divide \dimen2 by \count11
\multiply \dimen2 by #2 
\setbox0=\hbox{#3}\ht0=0pt\dp0=0pt
  \rlap{\kern\dimen1 \vbox to0pt{\kern-\dimen2\box0\vss}}\dimen1= \wd1
\dimen2=\ht1}
\def\pasgrille{
\count12= \dimen1 
\divide \count12 by 50
\divide \dimen2 by \count12
\count11 =\dimen2
\ 
\divide \dimen1 by 48
\setlength{\unitlength}{\dimen1}
\smash{\rlap{\ }}
\dimen1= \wd1
\dimen2=\ht1
}
\def\grille{
\count12= \dimen1 
\divide \count12 by 50
\divide \dimen2 by \count12
\count11 =\dimen2
\ 
\divide \dimen1 by 48
\setlength{\unitlength}{\dimen1}
\smash{\rlap{\graphpaper[1](0,0)(50, \count11)}}
\dimen1= \wd1
\dimen2=\ht1
}
\pasdegrille

\newcommand{\hamvf}{{H}}
\newcommand{\pa}{\partial}
\newcommand{\sK}{\mathscr{K}}
\newcommand{\abs}[1]{{\left\lvert{#1}\right\rvert}}
\newcommand{\norm}[1]{{\left\lVert{#1}\right\rVert}}
\newcommand{\ang}[1]{{\left\langle{#1}\right\rangle}}

\newcommand{\neigh}{\operatorname{neigh}}
\newcommand{\hph}{{\widehat \varphi}}
\newcommand{\Op}{{\operatorname{Op}^{{w}}_h}}
\newcommand{\be}{\begin{equation}}
\newcommand{\ee}{\end{equation}}

\newcommand{\CC}{{\mathbb C}}
\newcommand{\CI}{{\mathcal C}^\infty }
\newcommand{\CIc}{{\mathcal C}^\infty_{\rm{c}} }
\newcommand{\CIb}{{\mathcal C}^\infty_{\rm{b}} }

\newcommand{\RR}{{\mathbb R}}

\newcommand{\ST}{{\widetilde S_{\frac12}}}
\newcommand{\PT}{{\widetilde \Psi_{\frac12}}}

\newcommand{\defeq}{\stackrel{\rm{def}}{=}}

\newcommand{\NN}{{\mathbb N}}

\newcommand{\supp}{\operatorname{supp}}

\newcommand{\rest}{\!\!\restriction}

\renewcommand{\Re}{\mathop{\rm Re}\nolimits}
\renewcommand{\Im}{\mathop{\rm Im}\nolimits}
\newcommand{\ad}{\operatorname{ad}}

\newcommand{\ep}{\epsilon}

\theoremstyle{plain}

\newtheorem{thm}{Theorem}
\newtheorem{cor}{Corollary}
\newtheorem{prop}{Proposition}[section]
\newtheorem{lem}[prop]{Lemma}
\newtheorem*{dyn}{Dynamical Hypotheses}

\theoremstyle{definition}

\numberwithin{equation}{section}

\def\bbbone{{\mathchoice {1\mskip-4mu {\rm{l}}} {1\mskip-4mu {\rm{l}}}
{ 1\mskip-4.5mu {\rm{l}}} { 1\mskip-5mu {\rm{l}}}}}

\def\squarebox#1{\hbox to #1{\hfill\vbox to #1{\vfill}}} 
\newcommand{\stopthm}{\hfill\hfill\vbox{\hrule\hbox{\vrule\squarebox 
                 {.667em}\vrule}\hrule}\smallskip}

\newcommand{\tih}{{\tilde h}}

\usepackage{amsxtra}

\ifx\pdfoutput\undefined
  \DeclareGraphicsExtensions{.pstex, .eps}
\else
  \ifx\pdfoutput\relax
    \DeclareGraphicsExtensions{.pstex, .eps}
  \else
    \ifnum\pdfoutput>0
      \DeclareGraphicsExtensions{.pdf}
    \else
      \DeclareGraphicsExtensions{.pstex, .eps}
    \fi
  \fi
\fi

\title[Resolvent estimates for normally hyperbolic trapped sets]
{Resolvent estimates for normally hyperbolic trapped sets}
\author[J. Wunsch]
{Jared Wunsch}
\author[M. Zworski]
{Maciej Zworski}
\address{Northwestern University}
\address{University of California, Berkeley}

\begin{document}    


\maketitle   
\section{Introduction and statement of results}

We give pole free strips and estimates for resolvents
of semiclassical operators which, on the level of the 
classical flow, have normally hyperbolic smooth trapped
sets of codimension two in phase space.
Such trapped sets are structurally stable
-- see \S \ref{ind} -- and our motivation 
comes partly from considering the wave equation 
for slowly rotating Kerr black holes, whose trapped {\em photon spheres}
have precisely that dynamical structure -- see
\S \ref{inb}. From the semiclassical point of view an example to keep
in mind is given by 
\[  P ( z ) = -  h^2 \Delta + V ( x )-1  - z  \,, \ \ V \in \CI_c ( \RR^n ; \RR) 
\,, \]
with the classical flow described by Newton's equations:
\[  x' ( t ) = 2 \xi ( t ) \,, \ \ \xi' ( t ) = - V' ( x ( t ) ) \,, \ \ 
\varphi^t ( x( 0 ) , \xi ( 0 ) ) \defeq ( x (  t) , \xi ( t) ) \,. \]
The incoming and outgoing tails, $ \Gamma_\pm $, and the trapped set, $ K $, 
are defined by 
\[ \Gamma_\pm = \{ ( x , \xi ) \; : \; \exists \, M , \  
| \varphi^t ( x , \xi ) | \leq M \,, t \rightarrow \mp \infty \} \,, 
\ \ K \defeq \Gamma_+ \cap \Gamma_- \,. \]
As explained in \S \ref{inb} it is important to consider more 
general families of operator pencils. The general assumptions
will be given in \S \ref{ina} but the result is already non-trivial 
in the case presented above: $ X = \RR^n $, and $ P ( z ) = P - z $,
$ P = - h^2 \Delta + V ( x) - 1 $.

\begin{thm} 
\label{th1} Suppose that $ P ( z ) $ is a family of operators
satisfying the assumptions in \S \ref{ina}, with a trapped
set $K$ which is smooth and normally hyperbolic in the sense of \S
\ref{ind} and contained in $U_1 \Subset X.$

If the symbol of $ \partial_z P ( 0 ) $ is strictly negative near $ p^{-1} ( 0 ) 
\cap T_{U_2 }^* X $ and 
$ W \in \CI ( X ;\RR) $ satisfies 
\[ W \geq 0 \,, \ \ \  W \rest_{U_1 } = 0 \,, \ \ \ 
 W \rest_{ X \setminus U_2 } = 1 \,, \]
where $\pi(K) \subset \overline{U_1} \subset U_2 \Subset X$
then there exist $ \delta_0 ,\nu_0 > 0  $ such that  for $ | z| < \delta_0 $
we have 
\begin{equation}
\label{eq:th1} \|  ( P ( z) - i W ) ^{ -1 }\|_{L^2 \rightarrow L^2 }
\    \lesssim \ \begin{cases} 
1/ \Im z \,, & \Im z > 0  \,, \\
h^{-1} \log (1/h)\,, &   \Im z=0\,, \\ 
h^{-k} \,,  &    \Im z > -\nu_0 h  \,,
\end{cases}\end{equation}
and in particular, $ z \mapsto ( P ( z) - i W ) ^{-1} $ is holomorphic in 
$ \{ | z | < \delta_0 \,, \ \Im z > - \nu_0 h \}$.
\end{thm}

This result is related to the general principle in scattering theory
which in mathematics goes back at least to the work of Lax-Phillips
and Morawetz: the nature of trapping of rays is related to the
distance of \emph{resonances}, which is to say poles of the analytic
continuation of the resolvent, to the real axis. That in turn is
related to energy decay, local smoothing and other properties of the
propagators.  The closeness of these resonances to the real axis is
in particular related to the stability of the trapped trajectories, with stable
trapping giving rise to resonances close to the axis---heuristically,
these are close to being eigenvalues.  By contrast, trapped orbits
near which the dynamics is \emph{hyperbolic} leads to resonances
bounded away from the axis -- see \cite{ZwNo} for a general
introduction. In \cite{Ik} and \cite{NZ3} a gap was established when
hyperbolic trapped sets are fractal and a certain {\em topological
  pressure} condition is satisfied. In Theorem \ref{th1}
the trapped set is smooth and has the maximal dimension.  We assume
that the flow is $r$-normally hyperbolic for every $r$ on this trapped manifold in the
sense of Hirsch, Pugh, and Shub \cite{Hirsch-Pugh-Shub} and Fenichel
\cite{Fenichel}. That assumption is structurally stable -- see \S
\ref{ind}.

The proof of Theorem \ref{th1} is based on a positive commutator
argument with an escape function \eqref{eq:defG}
in a slightly exotic symbolic
class described in \S\S \ref{s2p}-\ref{eaq}. A similar
logarithmically flattened escape
function for more complicated (fractal) trapped sets
was used in \cite{SZ10}. For the semiclassical analysis near
closed hyperbolic orbits similar escape functions were
used by Christianson in \cite{Chr} and \cite{Chr2}. In a way,
the situation here is simpler as we assume that the 
trapped set has codimension two. However, following 
our arguments might simplify the treatment
of closed orbits as well. 

In Theorem \ref{th2} in \S \ref{rere} we present a closely 
related result for resonances. For operators $ P ( z ) = 
- h^2 \Delta + V ( x )-1  - z $ with $ V ( x ) $ 
holomorphic and decaying in a conic neighbourhood of $ \RR^n $ in 
$ \CC^n $ (in fact, for a larger class of operators
with real analytic coefficients in $ \RR^n $ -- \cite{HeSj}) 
a more precise resonance free region was obtained by 
G\'erard-Sj\"ostrand \cite{GeSj2}. The novelty in Theorems \ref{th1}
and \ref{th2} lies 
in the resolvent bounds and the applicability to $ \CI $ coefficients.
The estimates in microlocally weighted spaces of holomorphic functions
in \cite{GeSj2} do not immediately imply polynomial bounds in $ 
h $, in the resonance free strips. For more recent results 
involving scattering with hyperbolic trapped sets we refer
to \cite{BoAl},\cite{BoRa},\cite{NZ3},\cite{NZ4},\cite{PeSt},
and references given there.

As examples of immediate applications of Theorem \ref{th1}
we give the following corollaries which follow immediately from 
the results of \cite{Dat}:
\begin{cor}
\label{cor1}
Suppose that $ X $ is a scattering manifold (that is a manifold
with an asymptotically conic metric) and 
$ -\Delta_g $ is the non-negative Laplace-Beltrami operator
on $ X $. Suppose that the trapped set for the geodesic flow
on $ S^* X $ is normally hyperbolic in the sense of \S \ref{ind}. 
If $ r ( x ) = (1 + d ( x , x_0 )^2)^{\frac12} $, where $ 
d ( x, x_0 ) $ is the distance function to 
any fixed point $ x_0 \in X $, then for $ \lambda > 1 $
\[ \| r^{-\frac12-0} ( - \Delta_g - \lambda \pm i 0 )^{-1} r^{-\frac12-0} \|
_{L^2 ( X ) \rightarrow L^2 ( X ) } \leq  \frac  {C \log
\lambda } \lambda  \,. \]
\end{cor}
This implies local smoothing for the Schr\"odinger equation
with a tiny loss of regularity:
\begin{cor}
\label{cor2}
Under the assumptions of Corollary \ref{cor1} 
we have the following estimate valid for any $ T > 0 $ (large) and 
$ \epsilon, \delta > 0 $ (small):
\[ \int_0^T \|r^{-\frac12-\delta } \exp ( { it \Delta_g } ) u \|^2_{H^{\frac12
- \epsilon }  ( X ) }\, dt \leq C_{ T, \epsilon, \delta} \| u \|_{L^2( X ) }^2 
\,. \]
\end{cor}
Based on this, and assuming that the curvature of the 
asymptotically conic manifold is negative (in every compact set),
the results of \cite{BGH} show that Strichartz estimates hold
with no loss at all.

Our motivation for considering this geometric set-up comes from 
the \emph{Kerr black hole.}  This is a family of Lorentzian metrics
which solve the Einstein equations and describe rotating black holes. 
We refer to \cite{Dafermos-Rodnianski2} for a survey of
mathematical progress on the wave equation for these metrics and to 
\cite{Tataru-Tohaneanu1} for some more recent results and references. 
In the physics literature the decay of waves has been studied 
in terms of {\em quasinormal modes} which are the analogues of 
scattering resonances in this setting -- see \cite{KoSm} for
a physics introduction and \cite{BoHa} for a recent mathematical
result which provided an expansion of waves in the Schwarzschild
-De Sitter background in terms of resonances. 

Obstructions to rapid energy decay occur, heuristically, due to
separate mechanisms at high and low frequencies.  At high frequencies
it is expected that the geometry of the trapped set plays a key role
and it is on this geometry that we focus our attention.  As recalled
below, the trapped set of Kerr is indeed an $r$-normally hyperbolic
manifold (within the energy surface) for all $r$, diffeomorphic to
$T^* S^2$ (or $S^*S^2$ if we restrict to fixed energy).  It is thus of
interest to explore the limits placed on exponential local energy
decay by this trapping mechanism, and this is exactly the role of
resonances.  That is to say, as the Kerr metric is stationary, we may
Fourier transform away the ``time'' variable, and try to study the
poles of the putative analytic continuation of the resulting
stationary operator across its continuous spectrum.  This motivates
considering general operator pencils $ P ( z ) $ in place of $ P - z$.

In the case of Kerr the principal obstacle, compared to the
Schwarzschild analysis \cite{BoHa} is the failure of ellipticity of
the stationary operator $P ( z ) $ near the event horizon of the black
hole, within the so-called ``ergo-region.''  This failure reflects the
failure of our timelike Killing field (with respect to which we have
Fourier transformed) to be timelike in the region in question.  Thus,
we reduce our question to a simpler model problem by cutting away the
ergo-region.  To do this, we modify our stationary operator by
considering only the form of the operator near its trapped set, and
then adding a \emph{complex absorbing potential} to damp waves
propagating outward from it.  We then consider the complex eigenvalues
of the resulting non-self-adjoint operator as a proxy for resonances.
Such a construction is rigorously known to approximate resonances in
certain cases \cite{Stefanov1}.  Thus Theorem~\ref{th1} yields a gap
in the spectrum of the operator $P ( z ) -iW$ near the real axis, at
high frequency (i.e.\ in the semiclassical limit). Recently, a
meromorphic continuation of $ P( z )^{-1} $ and a rigorous definition
of quasinormal modes for Kerr-De Sitter black holes have been obtained
by Dyatlov \cite{Dya}.

Our paper is concerned only with the analysis near the trapped set.
Unlike in most 
other mathematical works on Schwarzschild and Kerr
black holes -- see for instance 
\cite{BS}, 
\cite{BoHa}, \cite{DSS,DSS1}, \cite{FKSY,FKSY:erratum},
\cite{SaBarreto-Zworski}, \cite{To}---this analysis of the trapped set
does not use separation of variables and properties of
the Regge-Wheeler potential. It  is carried out in a way
applicable  to the perturbations of the metric.
The structure of the trapped set does
not change under those pertubations but one cannot separate variables
anymore -- see the end of \S \ref{ind} and \S \ref{inb} for
more details.

To indicate how the local results near the trapped set 
can be used to obtain energy decay 
we present Theorem \ref{t:dec} in \S \ref{rere}. 
Here is its simplest version:
\begin{cor}
\label{cor3}
Suppose that $ X = X_0 \sqcup (\RR^n \setminus B ( 0 , R )) 
\sqcup \cdots \sqcup (\RR^n \setminus B ( 0 , R )) $, where
$ X_0 $ is a smooth compact Riemannian manifold with boundary, with 
the metric $ g$ equal to the usual Euclidean metric in the
infinite ends, $  \RR^n \setminus B ( 0 , R ) $. If $ n $ is 
odd and the trapped set for the geodesic flow on $ S^* X $ 
is normally hyperbolic in the sense of \S \ref{ind}, then 
the local energy decays exponentially: 
for any $ \epsilon > 0 $ there exists 
$ \alpha = \alpha ( \epsilon ) > 0 $, such that if
\[   (\partial_t^2 - \Delta_g ) u = 0 \,, \ \ u\rest_{t=0} = u_1 \,,
\ \ \partial_t u \rest_{t=0} = u_1 \,, \ \ \supp u_j \subset U 
\Subset X \,, \]
then for any $ V \Subset X $ we have
\begin{equation}
\label{eq:cor3}
 \int_{ V } \left( | u ( t , x )|^2 + |\partial_t u ( t, x ) |^2 
\right) 
dx  \leq C e^{ - \alpha t } (  \| u_0 \|_{H^{1+\epsilon} }^2 + 
\| u_1 \|_{H^\epsilon }^2 ) \,, \end{equation}
where $ C $ depends on $ U , V $, and $ \epsilon $.
\end{cor}

\medskip

\noindent
{\bf Comments on notation.} For a set $ A $ we denote 
by $ \neigh ( A ) $ a small open neighbourhood of $ A $.
For $ V $ a Banach space, $ f = {\mathcal O}_V ( g ) $ means
that $ \| f \|_V \leq C |g| $, with the similar notation 
for operators: $ T = {\mathcal O} ( g ) : V \rightarrow W $,
means $ \| T u \|_W \leq C | g | \|  u \|_{V }$. Unless
specified by a subscript $ C $ denotes a constant the value
of which may vary throughout the paper. The notation 
$ a \lesssim b $ means that $ a \leq C b $.

\medskip

\noindent
{\sc Acknowledgments.}
The authors gratefully acknowledge helpful conversations with Kiril
Datchev, St\'ephane Nonnenmacher, Clark Robinson, and Amie Wilkinson;
Semyon Dyatlov and Andr\'as Vasy provided helpful comments and
corrections to the manuscript.  This work was partly supported by NSF
grants DMS-0700318 (JW) and DMS-0654436 (MZ).

\subsection{Global assumptions on $ P ( z ) $.}
\label{ina}

We make abstract assumptions on $ P ( z ) $ in order to allow very
general end structures.  
The assumptions are in some sense the reversal of the black box
assumptions of \cite{BZ2} and \cite{SZ1}: we specify the operator in
the compact interaction region but allow an almost arbitrary structure
outside. That is natural since we are adding the complex absorbing
potential. Many results about resonances can be rephrased in this
setting.  In some cases they can then be ``glued'' to obtain global
results as was done for scattering manifolds in \cite{Dat}. Some
infinities appear remarkably resilient to that approach, in particular
the ends of conformally compact, that is asymptotically hyperbolic,
manifolds. However, we expect that the Kerr metrics can be ``glued''
to our local construction.

For a concrete example of operators satisfying the abstract
assumptions presented here see \S \ref{rere}.

We consider a holomorphic family of operators, 
\[   z \longmapsto P ( z ) \,, \ \ z \in D ( 0 , \delta_1 ) \,, \]
depending implicitly on the semiclassical parameter $ h $.
These operators act on 
 $\mathcal{H}$, a complex Hilbert space with an orthogonal decomposition 
$${\mathcal{H}}= L^2 ( X_0 ) \oplus {\mathcal{H}}_1 \,, $$
where $ X_0 \Subset X $ is an open submanifold of $ X $ with a 
smooth boundary. 

The corresponding orthogonal projections are denoted by  
$\bbbone_{0}u$ and $\bbbone_{1}u$ respectively, where $u\in \mathcal{H}$. 
The operators 
$$P(z):{\mathcal{H}}\longrightarrow{\mathcal{H}}$$
with the 
domain $\mathcal{D}$, 
independent of $z$ (and of the implicit parameter $ h $), 
and satisfying
$$\bbbone_{0} {\mathcal{D}}=H^2(X_0) \,, \ \ \ 
\partial_z P ( z ) : {\mathcal D} \longrightarrow {\mathcal H} \,, $$ 
see \cite{SZ1} for a more precise meaning of 
the first statement.

We also assume that  
\begin{equation}
\label{eq:2.2}
\bbbone_{0}P(z)u=P_0(z)(u \rest_{X_0} ),\; \text{ for } u\in
 {\mathcal{D}}\,, \end{equation}
where $P_0 (z) \in \Psi^{2,0}_h ( X ) $, 
for real values of $ z $, $ P_0 ( z ) $
is a formally self-adjoint operator on $L^2(X)$ given by  
\[ P = p ( x, h D ) + h p_1 ( x , h D ; h ) \,, \ \ 
p_1(x,hD) \in \Psi^{2,0}_h(X), \ \ p ( x, \xi ) \geq \langle \xi \rangle^2/C  - C 
\,; \]
see \S \ref{sc} for the definition of the classes of operators, and 
for the conditions on $ X$.

We assume that $ P ( z ) $ is self-adjoint for $ z \in \RR $, 
and that $ P ( z ) $ is holomorphic for $ z \in \CC $, $ | z| < \delta_1 $.
Hence, 
\[  P ( z ) = P ( \bar z )^* \,, \ \ |z| < \delta_1 \,, \ \ 
( P ( z ) - i )^{-1}  :  {\mathcal H} \longrightarrow {\mathcal D} \,, \ \
\Im z = 0 \,, \  |z| < \delta_1 \,.  \]
This
implies boundedness in a complex neighbourhood, since
$ P ( z ) - P ( \Re z ) = {\mathcal O} ( |\Im z | ) \; : \;
{\mathcal D } \rightarrow {\mathcal H} $:
\begin{equation}
\label{eq:nasty}
( P ( z ) - i )^{-1}  :  {\mathcal H} \longrightarrow {\mathcal D} \,, \ \
\text{for $ |z| < \delta_2 $.}
\end{equation}

The assumption that
\[  \bbbone_0 ( P - i)^{-1} : {\mathcal H} \longrightarrow {\mathcal H} 
\ \ \text{ is a compact operator,} \]
and estimates in \S \ref{ep} 
imply that $ ( P ( z ) - i W )^{-1} : 
{\mathcal H} \rightarrow {\mathcal H} $ is meromorphic in 
$ D ( 0 , \delta_1) $. However we do not make this assumption and
prove the estimates on the resolvent directly. 

As stated in Theorem~\ref{th1}, we further make local assumptions near
the trapped set as follows: the symbol of $ \partial_z P ( 0 ) $ is strictly negative near $ p^{-1} ( 0 ) 
\cap T_{U_2 }^* X $, and
$ W \in \CI ( X ;\RR) $ satisfies 
\[ W \geq 0 \,, \ \ \  W \rest_{U_1 } = 0 \,, \ \ \ 
 W \rest_{ X \setminus U_2 } = 1 \,, \]
where $\pi(K) \subset \overline{U_1} \subset U_2 \Subset X.$  Our
dynamical assumptions near $K$ follow in the next section.

Finally we will consider the operator with complex absorbing potential
given by
$$
P(z)-iW
$$
where we define the operator $W$ by
$$
\bbbone_0 W u = W(x) u\rest_{X_0}
$$
with $W(x)$ a smooth function equal to $0$ on $U_1$ and $1$ on $X_0
\backslash U_2,$ and 
$$
\bbbone_1 Wu= \bbbone_1 u.
$$

\subsection{Dynamical assumptions}
\label{ind}

We now discuss the dynamical hypotheses for Theorem~\ref{th1}.  We
first state the minimal hypotheses needed for the proof of the theorem
to apply.

Let $\varphi^t$ denote the flow generated by the Hamilton vector field $\hamvf_p.$
Let $r$ denote the distance function to a fixed point in $X$ and
locally define the backward/forward trapped sets by:
$$
\Gamma_{\pm}=\{\rho\in \pi^{-1}(U_2): \lim_{t \to \mp\infty} r(\varphi^t(\rho)) \neq \infty\}.
$$
Let $\Gamma_\pm^\lambda=\Gamma_\pm\cap p^{-1}(\lambda).$  We can then
define the \emph{trapped set}
$$
K=\Gamma_+\cap \Gamma_-
$$
and let $K_\lambda=K\cap p^{-1}(\lambda).$

\begin{dyn}

\mbox{}

\begin{enumerate}
\item There exists $\delta>0$ such that $dp \neq 0$ on $p^{-1}(\lambda)$ for $\abs{\lambda}<\delta.$

\item \label{gammas}$\Gamma_\pm$ are codimension-one 
smooth
manifolds intersecting transversely at $K.$ (It
is not difficult to verify that $\Gamma_\pm$ must then be coisotropic and
$K$ symplectic.)

\item \label{normflow} The flow is hyperbolic in the normal directions
to $K$ within the energy surface: there exist subbundles\footnote{The
  bundles $E^\pm$ may of course depend on $\lambda$ but we omit this
  dependence from the notation.} $E^\pm$ of
$T_{K_\lambda} (\Gamma_\pm^\lambda)$ such that 
$$
T_{K_\lambda} \Gamma_\pm^\lambda=T K_\lambda \oplus E^\pm,
$$
where $$d\varphi^t:E^\pm \to E^\pm$$ and
there exists $\theta>0$ such that for all $\abs{\lambda}<\delta,$
\begin{equation}
\label{eq:normh}
\| {d\varphi^t(v)} \|  \leq C e^{-\theta | t | } \| {v} \| \
\text{for all } v \in E^{\mp},\ \pm t \geq 0.
\end{equation}
\end{enumerate}
\end{dyn}

These assumptions can be verified directly for the trapped set of a
slowly rotating Kerr black hole (i.e.\ when $a$ is small) but they are
not stable under perturbations, hence do not obviously apply to
perturbations of Kerr.  However, we will show that Kerr in fact satisfies
a more stringent (and well-studied) hypothesis that \emph{is} stable
under perturbation, and that implies the Dynamical Hypotheses above.
In particular, the standard dynamical notion of
\emph{$r$-normal hyperbolicity} implies items \eqref{gammas} and \eqref{normflow}, and
\emph{is} stable under perturbations, modulo possible loss of derivatives:

Recall that the flow in the energy surface $p^{-1}(\lambda)$ near $K_\lambda$ is \emph{eventually absolutely $r$-normally
  hyperbolic} for every $r$ in the sense of \cite[Definition
4]{Hirsch-Pugh-Shub} if its time-one flow is a $\mathcal{C}^r$ map
preserving a manifold $\mathcal{K}_\lambda$ (which a priori need only
lie $\mathcal{C}^1$ but is then automatically in $\mathcal{C}^r$)
such that for all $ \rho \in K_\lambda$, there exists a splitting of the tangent
bundle into subbundles stable under the flow
$$
T_\rho p^{-1} ( \lambda)  = T_\rho K_\lambda \oplus E^+_\rho \oplus E^-_\rho
\,, \ \ \ 
d \varphi^t_\rho(E^\pm_\rho) = E^\pm_{\varphi^t(\rho)} \,,  
$$
and for each $r \in \NN$ there exist $\theta_0>0$ and $C>0$ (both depending on
$r$) such that for $t>0,$
\begin{equation}\label{rnormhyp}\begin{aligned}
 \sup_{\rho \in K_\lambda} \norm{d\varphi_\rho^t\rest_{TK_\lambda}}^r
 &\leq C e^{-t\theta_0} \inf_{\rho\in K_\lambda}
\norm{d\varphi_\rho^{-t}\rest_{E^+}}^{-1},\\
\inf_{\rho \in K_\lambda} \norm{d\varphi_\rho^{-t}\rest_{TK_\lambda}}^{-r}
&\geq C^{-1}e^{t \theta_0} \sup_{\rho \in K_\lambda} \norm{d\varphi_\rho^t\rest_{E^-}}
\end{aligned}
\end{equation}
with $\norm{\bullet}$ some (indeed, any) fixed Finsler metric.
This assumption thus entails not merely that there is expansion and contraction in the
normal direction to $K$ but also that this expansion/contraction is
considerably \emph{stronger} than any expansion and contraction occuring in the flow on
$K$ itself. 
We remark that one may easily check that \eqref{rnormhyp} is stronger than
\eqref{eq:normh} by noting that since $\varphi^t$ are all
diffeomorphisms, fixing a Riemannian metric gives $$\sup_{\rho \in TK_\lambda}
\norm{d\varphi^t(\rho)} \geq 1$$ for all $t,$ hence, for instance the
first line of \eqref{rnormhyp} gives the estimate \eqref{eq:normh}
for the bundle $E^+.$

We may replace hypotheses \eqref{gammas} and \eqref{normflow} with the
assumption that for $\abs{\lambda}<\delta,$ the trapped set
$K_\lambda$ has the property that the flow near it in
$p^{-1}(\lambda)$ is eventually absolutely $r$-normally hyperbolic for
every $r$.
The existence of manifolds
$\Gamma_\pm$ tangent to $E^\pm$ and satisfying the Dynamical
Hypotheses, as well as the structural stability of these assumptions,
are classical theorems of Fenichel \cite{Fenichel} and
Hirsch-Pugh-Shub \cite{Hirsch-Pugh-Shub}.  The resulting perturbed
stable/unstable and trapped manifolds are only finitely differentiable
in general, as $r$-normal hyperbolicity for each $r$ is the
structurally stable property, and this only entails $\mathcal{C}^r$
regularity; on the other hand this $r$ can be chosen as large as
desired.  While we stated the theorems above with $\mathcal{C}^\infty$
hypotheses for simplicity, it is manifest from the proofs that the
hypotheses could be reduced to insisting that $\Gamma_\pm$ be in
$\mathcal{C}^K$ for sufficiently large $K$, hence those results apply
to the perturbed trapped sets arising here.

Thus once we show in the following section that the trapped set for
Kerr satisfies the $r$-normal hyperbolicity assumptions, we will know
that perturbations of Kerr continue to satisfy the Dynamical
Hypotheses, with as much differentiability as is required.

\section{Trapping for Kerr black holes}
\label{inb}

The hypotheses in the preceding sections are motivated by the example
of the slowly rotating Kerr black hole.  In this family of examples, describing the
geometry of a rotating black hole, the structure of the trapped set is
as described above, while the global structure of the spacetime is
more complex. The proof that the Kerr trapped set is $r$-normally 
hyperbolic might be a new contribution. 

We now recall the Kerr geometry, and verify that the hypotheses from
the preceding section hold in a spatial neighbourhood of the trapped
set, at least for small values of the parameter $a$ describing the
angular momentum per unit mass of the black hole.

The Kerr metric is a metric given in
``Boyer-Lindquist'' coordinates by
$$
g=\frac{\Delta}{\rho^2}\big(dt - a \sin^2\theta d\varphi\big)^2-\rho^2
\big(\frac{dr^2}{\Delta}+d\theta^2\big) -\frac{\sin^2 \theta}{\rho^2}
\big( a dt -(r^2+a^2) \, d\varphi\big)^2,
$$
with
$$
\rho^2=r^2 +a^2 \cos^2 \theta,
$$
$$
\Delta=r^2-2Mr+a^2.
$$
We study this metric on $\RR\times (r_+,\infty)\times S^2$ with
$$
r_+ \defeq M+(M^2-a^2)^{1/2}.
$$
in this region, outside the ``event horizon'' $r=r_+,$ the metric is a nonsingular
Lorentzian metric.
The parameter $a \in [0,M)$ is the rotational parameter (angular
momentum per unit mass), and $M$ is the mass.
When $a=0$ we have spherical symmetry, and the Kerr metric
reduces to the Schwarzschild metric.

The d'Alembertian in the Kerr metric is given by
$$
\Box =  \big(\frac{(r^2+a^2)^2}{\Delta} - a^2 \sin^2 \theta \big)
\pa_t^2 + \frac{4 Mar}{\Delta} \pa_t \pa_\varphi +
\big(\frac{a^2}{\Delta}-\frac{1}{\sin^2 \theta} \big) \pa_\varphi^2 -\pa_r
\Delta \pa_r -\frac{1}{\sin\theta}\pa_\theta \sin \theta \pa_\theta.
$$

\renewcommand\thefootnote{\dag}%

Thus, setting $\Box u=0,$ if $u$ is of the form $e^{iE
  t}v_E(r,\theta,\varphi),$ we find that $v_E$ satisfies $ P_E v_E = 0 $,
where $ P_E $ is given by 
$$
 -E^2 \big(\frac{(r^2+a^2)^2}{\Delta} - a^2 \sin^2 \theta \big)
 + iE \frac{4 Mar}{\Delta} \pa_\varphi -
\big(-\frac{a^2}{\Delta}+\frac{1}{\sin^2 \theta}\big) \pa_\varphi^2 -\pa_r
\Delta \pa_r -\frac{1}{\sin\theta}\pa_\theta \sin \theta \pa_\theta
\,. 
$$
Setting $E=(1+hw)/h$ (and dropping the subscript on $v$) we have
\begin{multline}
\bigg( (1+2hw) \big(-\frac{(r^2+a^2)^2}{\Delta} + a^2 \sin^2 \theta \big)
 - (1+hw)\frac{4 Mar}{\Delta} hD_\varphi \\+
\big(-\frac{a^2}{\Delta}+\frac{1}{\sin^2 \theta}\big) (hD_\varphi)^2 +(hD_r)
\Delta (hD_r) +\frac{1}{\sin\theta}(h D_\theta) \sin \theta (hD_\theta)
+O(h^2) \bigg) v=0.
\end{multline}
Thus, if we set
\begin{multline}
\widetilde{P}=\big(-\frac{(r^2+a^2)^2}{\Delta} + a^2 \sin^2 \theta \big)
 -\frac{4 Mar}{\Delta} (hD_\varphi) \\+
\big(-\frac{a^2}{\Delta}+\frac{1}{\sin^2 \theta}\big) (hD_\varphi)^2 +(hD_r)
\Delta (hD_r) +\frac{1}{\sin\theta}(hD_\theta) \sin \theta (hD_\theta)+O(h^2)
\end{multline}
and
$$
\widetilde{Q}=2\big(\frac{(r^2+a^2)^2}{\Delta} - a^2 \sin^2 \theta \big)+\frac{4 Mar}{\Delta} (hD_\varphi),
$$
we are dealing with the equation
$$
(\widetilde{P}-hw \widetilde{Q})u=0.
$$
The operator $\widetilde{P}$ has disagreeable asymptotics near
the ends $r=r_+,\infty,$ however; we thus choose to
multiply our equation through by ${\Delta}/{r^4}$.
Thus, we let $P=(\Delta/r^4) \widetilde{P}$ and $Q=(\Delta/r^4) \widetilde{Q},$ so that
\begin{multline}
P=\big(-\frac{(r^2+a^2)^2}{r^4} + \frac{a^2\Delta}{r^4} \sin^2 \theta \big)
 -\frac{4 Ma}{r^3} (hD_\varphi) \\+
\big(-\frac{a^2}{r^4}+\frac{\Delta}{r^4 \sin^2 \theta}\big)
(hD_\varphi)^2 +\frac{\Delta}{r^4} (hD_r)
\Delta (hD_r) +\frac{\Delta}{r^4\sin\theta}(hD_\theta) \sin \theta (hD_\theta)+O(h^2)
\end{multline}
and
$$
Q=2\big(\frac{(r^2+a^2)^2}{r^4} - \frac{\Delta}{r^4} a^2 \sin^2 \theta \big)+\frac{4 Ma}{r^3} (hD_\varphi),
$$
and we are now interested to solutions of $P(z)u=0,$ with
$$
P(z)=P-zQ,\ z=hw.
$$
We are in the situation covered by Theorem~\ref{th1}
provided that we can verify the hypotheses on $P$ and $P'(0)=-Q.$
We note that $P$ and $Q$ are now self-adjoint with respect to the volume form
$$
\frac{r^4}{\Delta} \sqrt{\abs{g}} \, dr\, d\theta \, d\varphi.
$$
To see this, we write
$$
\widetilde{P}=\big(-\frac{(r^2+a^2)^2}{r^4} + \frac{a^2}{\Delta}{r^4} \sin^2 \theta \big)
 -\frac{4 Ma}{r^3} (hD_\varphi)+ P'
$$
where $P'$ is our original, formally self-adjoint operator $\Box,$
applied to functions independent of the $t$ variable (i.e.\ on the
quotient of the spacetime by the $\pa_t$ flow); the
$D_\varphi$ terms are self-adjoint by axial symmetry of $g.$

The hypotheses are, we claim,
satisfied in a subset $\{r>r_0\}$ (for some $r_0>r_+$) that includes the
trapped set and the $r\to+\infty$ end.  The hypotheses are not
globally satisfied, however, owing to the structure of $P$ near the
event horizon: not only is this end not asymptotically Euclidean, but
the operator $P$ is not even elliptic in a uniform neighbourhood of
$r=r_+:$ inside the ``ergosphere'' where
$$-\frac{a^2}{\Delta}+\frac{1}{\sin^2 \theta}<0,$$ $P$ is not elliptic
(i.e.\ the Killing vector field $\pa_t$ for the Kerr metric fails to
be timelike).  Thus we do not at this time know how to fit the global
structure of the Kerr metric into the assumptions made in \S\ref{ina};
for the moment we would instead have to consider a Kerr metric glued
to a Euclidean end in place of the $r\to r_+$ end.

In what follows, we verify that the structure of the Kerr trapped set, at least, is of
the desired form.
Letting
$$
\xi\,  dr + \alpha \, d\theta + \beta \, d\varphi
$$
denote the canonical one-form on $T^*X,$ we find that the
semiclassical principal symbol of $\widetilde{P}=(r^4/\Delta) P$ is\footnote{In our
  analysis of the null bicharacteristics, we study the operator
  $(r^4/\Delta)P$, which of course has no effect on the dynamics on
  $K_\lambda$.}
\begin{equation}\label{p}
p=\Delta\xi^2 +\alpha^2 + \big( \frac 1{\sin^2\theta} -
\frac{a^2}{\Delta}\big) \beta^2 - \frac{4Mar}{\Delta} \beta - \big(
\frac{(r^2+a^2)^2}{\Delta} -a^2 \sin^2 \theta\big)
\end{equation}
and the Hamilton vector field is given by
\begin{multline}\label{hamvf}
(1/2)\hamvf =\xi \Delta \pa_r + \alpha \pa_\theta- \frac{\big(a (a\beta + 2M
  r)-\beta \Delta\csc^2 \theta \big)}{\Delta} \pa_\varphi \\+ \big( \beta^2
\cot \theta \csc^2 \theta -a^2 \sin \theta \cos \theta \big)
\pa_\alpha\\ + \bigg((M-r)\xi^2+\frac{\big( a \beta (M-r) + r \Delta + M (a^2-r^2) \big)
  (a \beta + (a^2 +r^2) )}{\Delta^2}\bigg) \pa_\xi.
\end{multline}
We note (following Carter \cite{Carter}) that the quantities
$$
p,\ \beta,\ \text{and } \sK= \alpha^2 + \big(a \sin \theta - \frac
\beta{\sin \theta} \big)^2 
$$
are all conserved under the $\hamvf$-flow, and in involution, both on
\emph{and off} the energy surface $\{p=0\}.$

Under the $\hamvf$-flow, for each fixed $\beta,$ the sets of variables
$(\theta, \alpha)$ and $(r,\xi)$ evolve autonomously, with $\sK$
describing a conserved quantity in the $(\theta,\alpha)$ plane.
This demonstrates that the motion in the $(\theta,\alpha)$ variables is periodic.
Also,
$$
\sK-p = -2a \beta -\Delta\xi^2+ \frac{a^2 \beta^2+ 4Mar\beta + (r^2+a^2)^2}{\Delta}
$$
is conserved and (for $\beta$ fixed) dependent solely on $(r,\xi).$
This last observation means that in fact under the rescaled flow,
generated by $(1/2\Delta)\hamvf$, the
quantity
$$
-\dot r^2-2a \beta+ \frac{a^2 \beta^2+ 4Mar\beta + (r^2+a^2)^2}{\Delta}
$$
is constant.  For $a=0,$ this quantity is simply 
$$
-\dot r^2+\frac{r^4}{\Delta}.
$$
The ``potential'' $-\frac{r^4}{\Delta}$ has a nondegenerate local
maximum at $r=3M;$ this is its only critical point outside the event
horizon.  Thus this rescaled flow tends to $r=+\infty$ or $r=r_+$ except when
$r=3M,$ where it has an (unstable) invariant set $(r=3M, \xi=0).$
More generally, for $a$ small, the structure is more or less the same:
for each given $\beta,$ there is a unique local maximum of the
potential
$$
v_\beta(r)=2a \beta- \frac{a^2 \beta^2+ 4Mar\beta + (r^2+a^2)^2}{\Delta}
$$
outside $r=r_+.$ Thus, the trapped set $K$ consists of a family of
orbits on which $r=r(\beta), \xi=0,$ with $r(\beta)$ given by the
critical point of $v_\beta$ in the exterior of the black hole.  The
invariance of $p$ and $\beta$ on the four dimensional trapped set
$r=r(\beta), \xi=0$ with coordinates $(\theta,\varphi,\alpha,\beta)$
yields the desired integrability.  (Note that $p$ and $\beta$ are
manifestly in involution.)

To verify the hypothesis \eqref{eq:normh}, we note that since the
center manifold is given by $r=r(\beta),\xi=0,$ we need only verify
that the flow in $r,\xi$ is hyperbolic near these points.  The
linearization of this flow is simply
$$
\begin{pmatrix}
0 & \Delta(r) \\
B'(r) & 0
\end{pmatrix},
$$
where, by \eqref{hamvf},
$$
B(r)=\frac{\big( a \beta (M-r) + r \Delta + M (a^2-r^2) \big)
  (a \beta + (a^2 +r^2) )}{\Delta^2}
$$
The positivity of $B'(r)$ at $r=r(\beta)$ is equivalent to the
positivity of $A'(r),$ where
$$
A(r) = \big( a \beta (M-r) + r \Delta + M (a^2-r^2) \big)
  (a \beta + (a^2 +r^2)).
$$
When $a=0,$ strict positivity is easily verified at $r=r(\beta)=3M;$
again by perturbation, it persists for small $a.$

We note that in the special case of the Schwarzschild metric ($a=0$)
we can simply compute from \eqref{hamvf} that at the trapped set $r=3M,$ $\xi=0:$
$$
\begin{pmatrix}
(r-3M)' \\ \xi' 
\end{pmatrix}
= \begin{pmatrix}
0 & 3M^2\\ 9 & 0
\end{pmatrix}
\begin{pmatrix}
r-3M\\ \xi
\end{pmatrix}+ \mathcal{O}((r-3M)^2 +\xi^2),
$$
where primes denote derivatives under the flow generated by
$(1/2)\hamvf.$  Thus the unstable Liapunov exponent under the
$\hamvf$-flow is $6\sqrt{3} M.$

For any given $\beta,$ let $\gamma^\pm_\beta$ denote the subsets of
$\RR^2_{r,\xi}$ given by the stable and unstable manifolds of the
fixed point $(r=r(\beta),\xi=0).$  As $\beta$ is conserved under the
flow, the fibration
$$\{(r,\xi,\theta,\varphi,\alpha,\beta): (r,\xi) \in \gamma^\pm_\beta\}
\mapsto (r=r(\beta),\xi=0,\theta,\varphi,\alpha,\beta)$$
gives smooth fibrations of the stable and unstable manifolds of
the flow.  (The fibration is conserved under the flow since
$\gamma^\pm$ and $\beta$ are.)

To check the hypotheses on $Q=-P'(0),$ we note that
$$
\sigma(\widetilde Q) + p = \big(\frac{(r^2+a^2)^2}{\Delta} - a^2 \sin^2 \theta\big) +
\big(-\frac{a^2}{\Delta}+\frac{1}{\sin^2 \theta}\big) \beta^2+\text{
  nonnegative terms}.
$$
The first term on the right is bounded below by
$$
\frac{r^4 +a^2 r^2 + 2 M a^2 r}{\Delta}
$$
while the second is bounded below by
\begin{equation}\label{lowerbound}
\beta^2 \frac{r^2-2Mr}\Delta
\end{equation}
hence we obtain the positivity of $\sigma(\widetilde Q)$ (hence
negativity of $P'(0)$) in a spatial neighbourhood of the trapped set,
provided $a$ is not too large; recall that for $a=0,$ the trapped
set lies over $r=3M,$ where the latter term in \eqref{lowerbound} is
safely positive.

We now show that the hypotheses of Theorem~\ref{th1} are indeed
satisfied near the trapped set not just for the slowly rotating Kerr
metric itself, but for \emph{smooth perturbations} of such Kerr
metrics.  The crucial observation is that for $a$ small, the Kerr
metric is $r$-normally hyperbolic for every $r,$ and that these
properties are structurally stable, so that an invariant manifold
diffeomorphic to $S^*(S^2)$ persists, with the flow near it remaining
normally hyperbolic.  We recall that the perturbed trapped set may
cease to be infinitely differentiable: for any $r,$ a sufficiently
small perturbation gives a trapped set in $\mathcal{C}^r,$ but the
required perturbation size may shrink as $r\to\infty.$ In practice
this need not concern us, as the proof of Theorem~\ref{th1} only uses
a finite (albeit unspecified) number of derivatives.

\begin{prop}\label{normhyp}
  For $a$ sufficiently small, there exists a neighbourhood of $K,$ such
  that the flow generated by $\hamvf$ is $r$-normally hyperbolic for each
  $r,$ i.e.\ satisfies \eqref{rnormhyp}.  Hence, by the results of
  \cite{Hirsch-Pugh-Shub}, for each $r,$ any sufficiently small perturbation of the
  Kerr metric also gives rise to an $r$-normally hyperbolic
  trapped set (in $\mathcal{C}^r$) satisfying the hypotheses of \S\ref{ind}.
\end{prop}
\begin{proof}
We have verified above that $d\varphi^t\rest_{E^\pm}$ satisfies
$$
\| {d\varphi^t_\rho(v)} \|  \leq C e^{-\theta | t | } \| {v} \| \
\text{for all } v \in E^{\mp}_\rho,\ \pm t \geq 0,
$$
for some $\theta>0.$  To further verify \eqref{rnormhyp} we also require
estimates on $d\varphi^t\rest_{TK}.$ Recall that the flow on $K$ is
integrable for the simple reason that $p$ and $\beta$ are both
conserved (i.e.\ we only use axial symmetry here, not preservation of
$\sK$ as well).  Fixing the values of $p,\beta$ foliates $K$ into
invariant tori on which the flow is necessarily quasi-periodic.  As a
consequence of the quasi-periodicity, away from any possible
degenerate tori, we have action-angle variable $(I_1,\dots,I_n) \in
\RR^n,$ $(\theta_1,\dots \theta_n) \in (S^1)^n$ such
that $\hamvf=\sum \omega_j(I) \pa_{\theta_j},$ hence
$$
d(\varphi^t(\rho),\varphi^t(\rho'))^2 \sim \sum (I_j-I'_j)^2+
\big(\theta_j-\theta'_j + \big(\omega_j(I)-\omega_j(I')\big) t\big)^2 \lesssim d(\rho,\rho')^2(1+\ang{t}^2).
$$
Thus,
$$
\norm{d\varphi^t\rest_{TK}} \leq C \ang{t}.
$$

Near degenerate invariant tori, this argument breaks down, and could
in principle fail (e.g.\ there can be hyperbolic closed orbits on
surfaces of rotation).  However we claim that the same estimate in
fact holds globally on $K;$ it thus remains to check it near
degenerate tori.  Restricting $p$ given by \eqref{p} to the trapped
set, where $\xi=0$ and $r=r(\beta),$ we find that $dp$ and $d\beta$
are linearly dependent only at $\alpha=0,\ \theta=\pi/2,$ i.e.\ at the
equatorial orbits.  (A separate computation shows that orbits passing
through the poles, i.e.\ with $\beta=0$ are \emph{not} degenerate,
even though the coordinate system employed here is not valid near the
poles.)  Put another way, the functions $\beta$ restricted to the set
$K\cap p^{-1}(\lambda)$ has its only critical points along the set
$\alpha=0,\ \theta=\pi/2,\varphi \in S^1.$ In the case of the
Schwarzschild metric ($a=0$), there are two values of $\beta$ at which
this can occur, $\pm (E+r^2/\Delta)$ and they are respectively maxima
and minima nondegenerate in the sense of Morse-Bott.  In particular,
we may use coordinates $\alpha,\theta,\varphi$ on $K\cap
p^{-1}(\lambda),$ and for the Schwarzschild case, $K=\{r=3M\}$ and
$$
\beta=\pm \sin \theta \big(\lambda +\frac{r^4}{\Delta}-\alpha^2\big)^{1/2},
$$
hence at the critical manifold $\theta=\pi/2,\alpha=0$ we compute
$$
\beta''_{\alpha\alpha}=\mp (\lambda+27M^2)^{-1/2},\ \beta''_{\theta\theta}=\mp(\lambda +27M^2)^{1/2},\ \beta''_{\alpha\theta}=0.
$$
This establishes nondegeneracy, which extends by continuity of second
partial derivatives for the Kerr case when $a$ is small.

The behavior of an invariant torus in a three-dimensional energy surface
near a Morse-Bott maximum or minimum of a conserved quantity is well
understood (see, e.g.\ \cite{BMF}): it must be an invariant circle
surrounded by nondegenerate invariant tori shrinking down to it; in
particular, if $\beta$ takes on a maximum values $\beta_M,$ along an
equatorial orbit, then any sufficiently nearby orbit is constrained to lie for all
time in $\beta^{-1}((\beta_M-\epsilon,\beta_M)),$ and this is a solid
torus $S^1_\varphi\times B^2$ in the energy space surrounding the equatorial orbit $S^1_\varphi,$ whose
diameter can be made as small as desired by shrinking $\ep \to 0.$
Taking a cross section of this solid torus, we observe that the
Poincar\'e return map is thus a \emph{twist map} preserving the value
of $\beta,$ under whose
iterations the distances between points grows linearly in time.
Additionally, we of course have $\varphi'=\beta$ along the flow, so the
difference between $\beta$ values can grow at worst linearly along the
orbit.  Thus, we again obtain linear growth of distances along the orbit,
hence $d\varphi^t\rest_{TK}$ grows
at most linearly.  This implies \eqref{rnormhyp} for every $r.$
\end{proof}

We have thus established the dynamical hypotheses for the Hamilton
vector field $H,$ associated to $p.$ As $z \in \RR$ varies, this is
not all of the real part of the symbol of $P-zQ;$ by structural
stability, however, the hypotheses persist for the principal symbol of
$P-zQ$ for $z\in \RR$ sufficiently small.

Finally, we observe that in the the end of the manifold $r\to
+\infty,$ the assumptions on $P(z)$ can be routinely verified by use
of the \emph{semi-classical scattering calculus} of pseudodifferential
operators \cite{WZ}, as $P(z)-iW$ is elliptic in that setting.

\section{Analytic preliminaries}
\label{pr}

In this section we recall facts from semiclassical analysis
referring to \cite{DiSj} and \cite{EZB} for background material.

\subsection{Semiclassical calculus}
\label{sc}
Because of our assumptions, except in \S \ref{rere}, we will
only use semiclassical calculus on a compact manifold.
Thus, let $ X $ be a $\CI$ manifold which agrees with $ \RR^n $ outside a
compact set, or more generally has finitely many ends diffeomorphic to $\RR^n:$
\begin{equation}
\label{eq:X}
  X = 
X_0 \sqcup X_1 \sqcup \cdots \sqcup 
X_N \,, \text{ where } X_j = \RR^n \setminus B ( 0 , R ) \text{ for } j>0,
\text{ and } \ X_0 \Subset X.
\end{equation}
We introduce the class of semiclassical 
symbols on $ X $ (see for instance \cite[\S 9.7]{EZB}):
\[ S^{m,k} ( T^* X ) = \{ a \in \CI( T^* X \times (0, 1]  ) :
|\partial_x ^{ \alpha } \partial _\xi^\beta a ( x, \xi ;h ) | \leq
C_{ \alpha, \beta} h^{ -k } \langle \xi \rangle^{ m - |\beta| } 
\} \,, \]
where outside $ X_0 $ we take the usual $ \RR^n $ coordinates in this 
definition.
The corresponding  class of pseudodifferential operators is denoted by 
$ \Psi_{h}^{m,k} ( X ) $, and we have the quantization and symbol maps:
\[ 
\begin{split}
  & \Op \; : \; S^{m , k } ( T^* X ) \ \longrightarrow 
  \Psi^{m,k}_h ( X) \\
  & \sigma_h \; : \; \Psi_h^{m,k} ( X ) \ \longrightarrow 
  S^{ m , k } ( T^* X ) / S^{ m-1, k-1} ( T^* X ) \,, \end{split}
\]
with both maps  surjective,  and the usual properties 
\begin{gather}
\label{eq:psym}
\begin{gathered}
 \sigma_h ( A \circ B ) = \sigma_h ( A)\sigma _h ( B ) \,,\\
 0 \rightarrow \Psi^{m-1, k-1} ( X) \hookrightarrow \Psi^{m, k} ( X)
\stackrel{\sigma_h}{\rightarrow} S^{ m , k } ( T^* X ) 
/ S^{ m-1, k-1} ( T^* X ) \rightarrow 0 \,, 
\end{gathered}
\end{gather}
a short exact sequence, and
\[ \sigma_h \circ \Op : S^{m,k} ( T^* X ) 
\ \longrightarrow  S^{ m , k } ( T^* X ) / S^{ m-1, k-1} ( T^* X ) \,,
\]
the natural projection map.
The class of operators and the quantization map are defined locally using the
definition on $ \RR^n $:
\begin{equation}
  \label{eq:weyl}
  \Op (a) u ( x) = \frac{1}{ ( 2 \pi h )^n } 
  \iint  a \left( \frac{x + y }{2}  , \xi \right
  ) e^{ i \langle x -  y, \xi \rangle / h } u ( y ) 
dy d \xi \,. \end{equation}
  We remark only that when we consider the operators
  acting on half-densities we can define 
  the symbol map, $ \sigma_h $, onto
  $$  S^{ m , k } ( T^* X ) / S^{ m-2, k-2} ( T^* X ) \,. $$ 
We keep this in mind but 
for notational simplicity we suppress the half-density notation. 

For future reference, and to illustrate the uses of the calculus,
we present the following application:
\begin{prop}
\label{p:work}
Suppose $ P \in \Psi^{2,0}_h  ( X ) $ satisfies 
$ P = p ( x, h D ) + h p_1 ( x , h D ; h ) $, $ p_1 \in \Psi^{2,0}_h $,
$ p ( x, \xi ) \geq \langle \xi \rangle^2/C  - C $. 

\medskip
\noindent
(i)
Let $ \psi_j \in \CIb (T^*X ; [ 0 , 1 ] ) $, $ j = 1,2$,  satisfy
\[ \psi_j  = 1 \ \text{in } \ { p^{-1} ( [- j \delta, j \delta ]) ) }= 1 \,, \ \
\supp \psi_j \subset p^{-1} ( [ - (j+1/2) \delta, (j+1/2) \delta ] ) \,.
\]
Then there exists $ E_1 \in 
\Psi^{-2, 0}_h ( X ) $, 
such that  
\[  E_1 \circ P = I + R_1 \,, \ \  R_1 \in \Psi^{0,0}_h ( X ) \,, \]
and 
\[ ( 1 - \psi_2^w ( x , h D ) ) R_1 \in \Psi^{-\infty, -\infty}_h ( X )  \,,\ 
 \  \psi_1^w ( x, h D ) E_1 
\in \Psi^{-\infty, -\infty}_h ( X ) \,. \]

\medskip
\noindent
(ii) Suppose $ f \in \CI_b ( X , [0,1] ) $, satisfies 
$ f \equiv 1 $ on $ U \subset X $, $ U $ open. 
Then there exists
$ E_2 \in \Psi^{-2, 0}_h ( X ) $, 
such that  
\[  E_2 \circ (  P - i f ) = I + R_2 \,, \ \  R_2 \in \Psi^{0,0}_h ( X ) \,, \]
and 
\[ \chi R_2 \,, \ R_2 \chi \in \Psi_h^{-\infty, -\infty } ( X ) \,, \ \ 
\text{for any $ \chi \in \CIc ( X ) $, \ $ \supp \chi \Subset U $.}\]
\end{prop}

\subsection{$S_{\frac12} $ spaces with two parameters.}
\label{s2p}

As in \cite[\S 3.3]{SZ10} we define the following symbol class:
\begin{equation}
\label{eq:sthf}
a \in S^{m , {\widetilde m}, k }_{\frac12} ( T^* \RR^n ) 
\ \Longleftrightarrow \ |\partial_x^\alpha \partial^\beta_\xi 
a (  x, \xi ) | \leq C_{\alpha \beta} h^{-m } \tilde h^{-{\widetilde m}} 
\left( \frac{\tilde h }{ h } 
\right)^{\frac12 ( |\alpha | + |\beta| ) } \langle \xi \rangle^{k-|\beta|}
 \,,\end{equation}
where in the notation we suppress the dependence of $ a $ on $ h 
$ and $ \tilde h $.
When working on $ \RR^n $ or in fixed local coordinates we will
use a simpler class
\begin{equation}
\label{eq:sthfs}
a \in \ST  ( T^* \RR^n ) 
\ \Longleftrightarrow \ |\partial^\alpha
a | \leq C_{{\alpha}, N }  \left( {\tilde h }/{ h }
\right)^{\frac12  |\alpha | } \langle \xi \rangle^{-N} \,. \\
\end{equation}
Then standard results (see \cite[\S 9.3]{EZB}) show that if  
$ a \in   S^{m , {\widetilde m},k }_{\frac12}$ and
$ b \in   S^{m', \widetilde m', k'}_{\frac12} $ then 
\[ a ( x , h D_ x ) \circ b ( x , h D_ x ) = c ( x , h D_ x  ) 
\ \text{ with } \ 
c \in   S^{m +m', {\widetilde m}+ {\widetilde m}' , 
k+ k' }_{\frac12}\,.
\]
The presence of the additional parameter $ \tilde h $ allows us to 
conclude that 
\[ c \equiv \sum_{ |\alpha | <  M } \frac{1}{\alpha !} 
\partial_\xi^\alpha a D_x^\alpha b \ \mod S^{m  + m ' , {\widetilde m} + {\widetilde m}' 
- M , k + k' - M }_{\frac12} \,, \]
that is, we have a symbolic expansion in powers of $ \tilde h $. 
We denote our class of operators by $ \Psi_{\frac12}^{m , {\widetilde m}, k } 
(T^* \RR^n ) $, or in the case of symbols in $ \ST $, $ \PT $.

A standard rescaling shows that this class of pseudodifferential 
operators is essentially equivalent to the calculus with a new
Planck constant $ \tih$: put
\begin{equation}
\label{eq:resc}
  ( \tilde x , \tilde \xi ) = ( \tilde h / h )^{\frac12} ( x , \xi ) \,,
\end{equation}
and define the following unitary operator on $ L^2 ( \RR^n ) $:
\[ U_{ h , \tilde h } u ( \tilde x ) = ( \tilde h / h )^{\frac{n}{4}}
u ( ( h / \tilde h )^{\frac12}  \tilde x ) \,.\]
The one easily checks that
\[ a ( x , h D_ x ) = U_{ h , \tilde h }^{-1} a_{ h , \tilde h }
( \tilde x , \tilde h D_{\tilde x } ) U_{ h , \tilde h} \,, \ \
a_{ h , \tilde h } ( \tilde x , \tilde \xi ) = a (
( h / \tilde h)^{\frac12} ( \tilde x , \tilde \xi ) ) \,. \]
Clearly $ a $ satisfies \eqref{eq:sthfs} if and
only if
$ a_{ h, \tilde h } \in S ( T^* \RR^n ) $, with estimates uniform
with respect to $ h $ and $ \tih$.

We recall \cite[Lemma 3.6]{SZ10} which provides explicit 
error estimates on remainders.
\begin{lem}
\label{l:err}
Suppose that 
$ a, b \in \ST $, 
and that $ c^w = a^w \circ b^w $. 
Then 
\begin{equation}
\label{eq:weylc}  c ( x, \xi) = \sum_{k=0}^N \frac{1}{k!} \left( 
\frac{i h}{2} \sigma ( D_x , D_\xi; D_y , D_\eta) \right)^k a ( x , \xi) 
b( y , \eta) \rest_{ x = y , \xi = \eta} + e_N ( x, \xi ) \,,
\end{equation}
where for some $ M $
\begin{equation}
\label{eq:new1}
\begin{split}
& | \partial^{\alpha} e_N | \leq C_N h^{N+1}
 \\
& \ \ 
\times \sum_{ \alpha_1 + \alpha_2 = \alpha} 
 \sup_{ 
{{( x, \xi) \in T^* \RR^n }
\atop{ ( y , \eta) \in T^* \RR^n }}} \sup_{
|\beta | \leq M  \,, \beta\in \NN^{2d} }
\left|
( h^{\frac12} 
\partial_{( x, \xi; y , \eta) } )^{\beta } 
(i  \sigma ( D) / 2) ^{N+1} \partial^{\alpha_1} a ( x , \xi)  
\partial^{\alpha_2} b ( y, \eta ) 
\right| \,,
\end{split} 
\end{equation}
where $ \sigma ( D) = 
 \sigma ( D_x , D_\xi; D_y, D_\eta ) \,. $ 
\end{lem}

As a particular consequence we notice that if $ a \in 
\ST ( T^* \RR^n)  $ and $ b \in  S ( T^* \RR^n ) $ then 
\begin{gather}
\label{eq:abc}
  c ( x, \xi) =
 \sum_{k=0}^N \frac{1}{k!} \left( 
i h \sigma ( D_x , D_\xi; D_y , D_\eta) \right)^k a ( x , \xi) 
b( y , \eta) \rest_{ x = y , \xi = \eta} + {\mathcal O}_{\ST}
( h^{\frac{N+1}2 } \tilde h^{\frac{N+1}2} ) 
\,. 
\end{gather}

\subsection{The $ \PT $ calculus on a manifold}
\label{ptm}

On a manifold of the type defined in the beginning of
\S \ref{sc} we consider the following class $ \ST  $:
\[ \ST = \ST ( T^*X ) \defeq
\{ a \in \CI ( T^*X ) \; : \;
\partial_{(x,\xi)}^\alpha a =  (h/\tih)^{ - |\alpha |/2 }
{\mathcal  O} ( \langle \xi \rangle^{-\infty} ) \} \,, \]
where outside of a compact set we use Euclidean coordinates,
determined by the infinite ends of $ X $. 

We first observe that this class is invariant under 
symplectic lifting of 
diffeomorphisms of $ X $, constant outside 
of a compact set. To define $ \PT ( X ) $ we 
need to check invariance of $ \ST ( T^* \RR^n ) $ under 
local changes of coordinates. Towards that we have
the following lemma:
\begin{lem}
\label{l:inv}
Suppose that $ a \in \ST ( T^*\RR^n ) $, $ U_j \subset 
\RR^n $, $ j = 1, 2 $ are open, and $ f : U_1 \rightarrow U_2 $ is a 
diffeomorphism. Let $ \chi \in \CIc ( U_1 ) $. 
Then $ A_2 \defeq \chi a^w ( x, h D ) \chi = a_\chi ( x, h D ) $, 
where $ a_\chi \in \ST $, $ a_\chi = \chi a \chi + {\mathcal O}_\ST 
( h^{\frac 12 } \tih^{\frac 12} )  $. For 
$   A_1 \defeq (f^{-1})^* A f^* $, 
we have 
\[  A_1 = a_f^w ( x, h D ) \,,  \ \ a_f \in \ST ( T^* \RR^n ) \,, \]
and 
\begin{equation}
\label{eq:change}
  a_f ( x , \xi  ) = \chi ( f^{-1} ( x ) )  a ( f^{-1} ( x ) ,
{}^tf'(x) \xi ) \chi ( f^{-1} ( x ) ) + {\mathcal O}_{ \ST } 
( h^{\frac 12 } \tih^{\frac 12} ) \,. 
\end{equation}
\end{lem}

\medskip
\noindent
{\bf Remark.} It seems important that 
we use the Weyl quantization. 
In the case of the right quantization 
\[ a^1 ( x, h D ) u = \frac 1 { ( 2 \pi h)^n }
\int \int a^1 ( x , \xi ) e^{ i \langle x - y , \xi \rangle/h } 
u ( y ) dy d \xi \,, \]
we have the exact formula
\[ a_f^1 ( f (  x)  , \eta ) = e^{-i \langle f ( x) , \eta 
\rangle/h  } a^1_\chi ( x , h D )  e^{ i \langle
 f ( x ) , \eta \rangle / h } \,, \]
see \cite[(18.1.28)]{Hor}. The asymptotic expansion
\cite[(18.1.30)]{Hor}, 
\begin{gather*}
 a_f^1 ( f ( x ) , \eta ) \sim 
\sum_{ \alpha \in N^n } \frac 1 {\alpha !} 
( \partial_\xi^{\alpha} a^1_\chi )  ( x , {}^t f'(x) \eta ) ( h D_y)^{\alpha} e^{ 
i\langle \rho_x ( y ) , \eta \rangle /h } \rest_{x=y} \,, \\ 
\rho_x ( y ) \defeq f ( y ) - f ( x ) - f'(x) ( y - x ) \,, 
\end{gather*}
is valid in our case {\em as an expansion in $ \tih $ only}.
In fact, due to the second order of vanishing of
$ \rho_x $ at $ x $, 
\[ ( h D_y)^{\alpha} e^{ 
i\langle \rho_x ( y ) , \eta \rangle /h } \rest_{x=y} = 
{\mathcal O} ( h^{|\alpha|/2} \langle \eta \rangle^{|\alpha|/2} ) \,, \]
and 
\[ ( \partial_\xi^{\alpha} a^1)  ( x , {}^t f'(x) \eta ) 
= {\mathcal O} ( ( h/\tilde h )^{-|\alpha|/2} \langle \eta \rangle^{-\infty}
 \,. \]
Hence the terms in the expansion are in 
$$  \tih^{|\alpha|/2 }\ST $$
(the term with $ |\alpha | = 1 $ vanishes).

The Weyl quantization will also be important
in local arguments in \S \ref{ere}. 
Finally we remark that for this class of symbols the improvement 
in the error occurs only in $ \tih $
when the action of half-densities is considered -- see
\cite[Appendix]{SZ8} or \cite[Theorem 9.12]{EZB}.

\medskip

\begin{proof} The statement about $ a_\chi $ follows from 
Lemma \ref{l:err}. For the change 
of variables we consider the Schwartz kernels of $ A_2 = a^2_\chi ( x, h D ) $
and $ A_1 =  a_f^w ( x, h D ) $ as densities:
\begin{equation}
\label{eq:weyla}
  K_b ( x , y ) | d y | \defeq 
 \frac{1}{ ( 2 \pi h )^n } 
\int   b \left( \frac{x + y }{2}  , \xi \right
) e^{ i \langle x -  y, \xi \rangle / h }  d \xi
|dy |   \,, \end{equation}
which means we seek $a_f$ such that
\begin{equation}
\label{eq:inva}
  K_{a_\chi}  ( x , y ) | d y | =
  K_{ a_f } ( \tilde x , \tilde y ) | d \tilde y| \,, \ \ \ 
\tilde x = f ( x ) \,,  \ \ \tilde y = f ( y ) \,. 
\end{equation}
We 
rewrite the right-hand side as 
by changing variables
\[ \frac{1}{ ( 2 \pi h )^n } 
\int   a_f \left( \frac{ f (  x ) + f(  y)  }{2}  , \tilde 
\xi \right
) e^{ i \langle f(  x)  -  f(  y) , \tilde 
\xi \rangle / h }  d \tilde \xi  | f ' ( y ) | |d  y |    \,. \]
Writing,
\begin{gather}
\label{eq:kur}
\begin{gathered}
 f ( x) - f ( y ) =  F ( x , y ) ( x - y ) \,, \ \ 
F ( x, y ) = f' \left( \frac{x+y}{2} \right) + 
{\mathcal O} ( ( x-y)^2 ) \,, \\ 
f( x  ) + f ( y ) = f \left( \frac{x+y}{2} \right) + 
{\mathcal O} ( ( x - y ) ^2 ) \,.
\end{gathered}
\end{gather}
we apply the ``Kuranishi trick'' by changing 
variables in the integral, $ \xi = F ( x , y )^{t} \tilde \xi $:
\[
\begin{split}
&  \frac{1}{ ( 2 \pi h )^n } 
\int  \left(  a_f \left( f \bigg(
\frac{   x +   y  }{2} \bigg)  ,  ( F( x, y )^t)^{-1} \xi \right) 
+ {\mathcal O}_{ \ST } \left(   { \tih}^{\frac 12}  h  
^{ -\frac 12 } ( x - y )^2 \right) \right) 
 \\
& \ \ \ \ \ \ \ \ \ \ \ \ \  \times e^{ i \langle   x  -    y ,  \xi 
\rangle / h }  d  \xi   | F( x , y )^t  |^{-1} 
| f ' ( y ) | |d  y |   \\
& = \frac{1}{ ( 2 \pi h )^n } 
\int  \left(  a_f \left( f 
\bigg(\frac{    x +   y  } {2} \bigg) ,   \left(  f' \bigg( 
\frac {x+ y } 2 \bigg)^t \right) ^{-1} \xi \right) 
+ {\mathcal O}_{ \ST }\left( { \tih}^{\frac 12}  h  
^{ -\frac 12 }  ( x - y )^2 \right) \right)
\\
& \ \ \ \ \ \ \ \ \ \ \ \ \  \times e^{ i \langle   x  -    y ,  \xi 
\rangle / h }  d  \xi  | f'(  ( x + y )/2 ) |^ {-1}  
| f ' ( y ) | |d  y |  \,.  \end{split} \]
We now observe that
\[  | f'(  ( x + y )/2 ) |  = | f ' ( y ) | 
+ {\mathcal O} (  |x - y |  ) \,, \]
and  consequently  $   K_{a_f } ( \tilde x , \tilde y ) | d \tilde y | = \ $
\[ \begin{split} 
& \frac{1}{ ( 2 \pi h )^n } 
\int  \left(  a_f \left( f \bigg(
\frac{   x +   y } {2} \bigg) ,   \left(  f' \bigg( 
\frac {x+ y } 2 \bigg)^t \right) ^{-1} \xi \right) 
 + 
{\mathcal O}_{ \ST }\left( { \tih}^{\frac 12}  h  
^{ -\frac 12 }  ( x - y )^2  +  | x - y |  \right)  
 \right) \\
& \ \ \ \ \ \ \ \ \ \ \ \ \ \ \ \ 
 e^{ i \langle   x  -    y ,  \xi 
   \rangle / h }  d  \xi  |d  y | \,. \end{split}
\]
The terms 
\[ {\mathcal O}_{ \ST }\left( { \tih}^{\frac 12}  h  
^{ -\frac 12 }  ( x - y )^2  \right) \]
contribute terms  $ {\mathcal O}_{\ST}  (\tih^{\frac 32} h^{\frac12} ) $
to the symbol: we use 
integration by parts based on 
$$ (x - y ) \exp ( \langle x - y , \xi 
\rangle / h ) = h D_\xi \exp ( \langle x - y , \xi 
\rangle / h ) \,. $$
 Similarly, smooth terms of the form 
$  {\mathcal O}_{ \ST }\left(| x - y | \right)  $ 
give contributions of the form 
 $ {\mathcal O}_{\ST}  (\tih^{\frac 12} h^{\frac12} ) $. 
Here in dealing with the ``big-Oh'' terms we use the fact that for $ b = b( x, y , \xi )  \in \ST $
(with the definition modified to include derivatives with respect
to $ y $),
\[  \frac{1}{ ( 2 \pi h )^n } 
\int b ( x, y , \xi )  e^{ i \langle   x  -    y ,  \xi 
   \rangle / h }  d  \xi   
=  \frac{1}{ ( 2 \pi h )^n } 
\int b_w \left( \frac{ x+ y } 2 , \xi \right)  
e^{ i \langle   x  -    y ,  \xi 
   \rangle / h }  d  \xi   \,, \]
where 
\[  b_w ( x, \xi ) = b ( x , x , \xi ) + {\mathcal O}_\ST 
( \tih ) \,,\]
which follows from the standard pseudodifferential  calculus 
and the rescaling \eqref{eq:resc}.

This shows that $   K_{a_f } ( \tilde x , \tilde y ) | d \tilde y | =  \ $
\[ 
 \frac{1}{ ( 2 \pi h )^n } 
\int  \left(  a_f \left( f \left(
\frac{   x +   y } {2} \right) ,   \left(  f' \left( 
\frac {x+ y } 2 \right)^t \right) ^{-1} \xi \right) 
 + 
{\mathcal O}_{ \ST }\left( { \tih}^{\frac 12}  h  
^{ \frac 12 } )  \right)  \right) e^{ i \langle   x  -    y ,  \xi 
   \rangle / h }  d  \xi  |d  y | \,,
\]
hence $a_f$ can be chosen in the form \eqref{eq:change} so that this matches
$ K_a ( x , y ) | d y |.$
\end{proof}

We need one more lemma which shows that away from the diagonal
the symbol contribution is negligible in $ h$ (rather
than merely in the $ \tih $ sense). This does not contradict the 
rescaling \eqref{eq:resc} which eliminates $ h $, as the
distance to the diagonal then grows proportionally to $ h^{-1/2} $
(see \cite[Theorem 4.18]{EZB}).

\begin{lem}
\label{l:loc} 
Suppose that $ \chi_j \in \CIc ( \RR^n ) $ are independent of 
$ h $, and $ \supp \chi_1 \cap \supp \chi_2 = \emptyset $. If
$ a \in \ST ( T^* \RR^n ) $ then 
\[  \chi_1 a^w ( x, h D ) \chi_2 = {\mathcal O}_{ {\mathcal S}' 
\rightarrow {\mathcal S} } ( h^\infty ) \,. \]
\end{lem}
\begin{proof}
We can apply Lemma \ref{l:err} as in the composition 
formula for $ a \in \ST $ and $ b \in S $ presented 
in \eqref{eq:abc}: in the composition $  \chi_1 a^w \chi_2 $
all terms in the expansion vanish and the error becomes
arbitrarily smoothing and bounded by $ h^{N } $, for any 
$ N $.
\end{proof}

Using Lemmas \ref{l:inv} and \ref{l:loc} 
we obtain an invariantly defined symbol
map for the class $ \PT( X ) $ defined using local coordinates,
as in \cite[\S 18.2]{Hor} (see \cite[\S E.2]{EZB} for
the semiclassical case). The symbol map occurs in the following
short exact sequence:
\[ 0 \longrightarrow h^{\frac 12} \tih^{\frac12 } \PT ( X ) 
\longrightarrow \PT ( X ) \stackrel{\tilde \sigma_{\frac12}}{
\longrightarrow } \ST ( T^* X ) / h^{\frac 12} \tih^{\frac12 } \ST ( T^*X ) 
\longrightarrow 0  
\,. \]

This means that if we start with $ a \in h^{-m} \ST ( T^* X ) $ then 
the operator $ a^w ( x, h D) \in h^{-m} \PT ( X ) $ is well defined
and its symbol is determined in any local coordinates up to terms
in $ h^{-m+\frac12} \tih^{\frac12} \ST $. We will be particularly 
interested in the case 
\begin{equation}
\label{eq:STd} 
   a \in \ST^{-} ( T^* X ) \defeq \bigcap_{m>0}  h^{-m} \ST ( T^* X ) \,, 
\end{equation}
in which case the local symbols will be determined up to terms
of size $ h^{\frac12 } \tih^{\frac 12} \ST^- $.

\subsection{Exponentiation and quantization}
\label{eaq}

As in \cite{SZ10} and \cite{Chr} it will be important to consider 
operators $ \exp G^w ( x, h D ) $, where $ G \in \ST^- $. 
To understand conjugated operators, 
\[ \exp( {-G^w ( x ,h D) } ) P \exp( { G^w ( x, h D) } ) \,, \]
we will use 
a special case of a result of
Bony and Chemin \cite[Th\'eoreme
6.4]{BoCh} -- see \cite[Appendix]{SZ10} or \cite[\S 9.6]{EZB}.
Because of the invariance properties established in \S \ref{ptm}
we discuss only the case of $ \RR^n $ in the next two subsections.

Let $ m ( x , \xi ) $ be an order function in 
the sense of \cite{DiSj}:
\begin{equation}
\label{eq:orderm}
  m ( x , \xi ) \leq C m ( y , \eta ) \langle  ( x-y , \xi - \eta ) 
\rangle ^N \,. 
\end{equation}
The class of symbols, $ S ( m ) $,  corresponding to $ m $ is defined as
\[  a \in S ( m ) \ \Longleftrightarrow \ |\partial^\alpha_x \partial_\xi^\beta
a ( x , \xi ) | \leq C_{\alpha \beta} m ( x , \xi ) \,. \]
If $ m_1$ and $ m_2 $ are order functions in the sense of \eqref{eq:orderm},
and $ a_j \in S (m_j ) $ then (we put $ h=1 $ here),
\[ a_1^w ( x, D ) a_2^w ( x , D ) = b^w ( x , D) \,, \ \ b \in S ( m_1 m_2 ) \,,\]
with $ b $ given by the usual formula,
\begin{gather}
\label{eq:usual}
\begin{gathered} \begin{split} b ( x , \xi ) & = a_1 \; \# \; a_2 ( x, \xi ) \\
& \stackrel{\rm{def}}{=}
\exp ( i \sigma ( D_{x^1}, D_{\xi^1} ; D_{x^2 } , D_{\xi^2} )/2  ) 
a_1 ( x^1 , \xi^1 ) a_2 ( x^2 , \xi^2 ) \rest_{ x^1 = x^2 = x , 
\xi^1 = \xi^2 = \xi } \,. \end{split} \end{gathered}
\end{gather}
A special case of \cite[Th\'eoreme 6.4]{BoCh} (see \cite[Appendix]{SZ10}) gives 
\begin{prop}
\label{p:bbc}
Let $m $ be an order function in the sense of \eqref{eq:orderm} and
suppose that $ G \in \CI ( T^* \RR^n; \RR  )  $ satisfies 
\begin{equation}
\label{eq:bc1}
 G ( x , \xi ) - \log m ( x , \xi )  = {\mathcal O} ( 1 ) \,, \ \
\partial_x^\alpha \partial_\xi^\beta G ( x, \xi ) = {\mathcal O} ( 1 ) \,,
 \ \ |\alpha | + |\beta| \geq 1 \,.
\end{equation}
Then 
\begin{equation}
\label{eq:bc2} 
\exp ( t G^w ( x , D ) ) = B_t^w ( x , D ) \,, \ \ B_t \in S ( m^t ) \,.
\end{equation}
Here $ \exp ( t G^w ( x , D) ) $ is constructed 
by solving $ \partial_t u = G^w ( x , D ) u $, $ u \in {\mathcal S} $. 
The estimates on $ B_t \in S( m^t ) $ depend 
{\em only} on the constants in \eqref{eq:bc1} and in \eqref{eq:orderm}.
In particular they are independent of the support of $ G $.
\end{prop}
Since $ m^t $ is the order function $ \exp ( t \log m ( x , \xi ) ) $, 
we can say that on the level of order functions ``quantization 
commutes with exponentiation''.

\subsection{Conjugation by exponential weights} 
\label{cew}

Let $ m $ be an order function for the $ \ST $ class:
\[   m ( \rho ) \leq C m ( \rho' ) \left\langle \frac{ \rho - \rho' }
{ ( h/\tilde h)^{\frac12} } \right\rangle^N \,, \]
for some $ N $.  We will consider order functions satisfying
\begin{equation}
\label{eq:ord}  m \in \ST ( m ) \,, \ \ \frac{1} m \in 
\ST \left ( \frac 1 m \right) \,.
\end{equation}
This is equivalent to $ m ( \rho ) = \exp G ( \rho ) $ with 
\begin{gather}
\label{eq:ord1}
\begin{gathered}
   \frac{ \exp G ( \rho ) } { \exp G ( \rho' ) } 
 \leq C \left \langle \frac{ \rho - \rho' }
{ ( h/\tilde h)^{\frac12} } \right\rangle^N \,, \ \ \ 
\ \ \partial^\alpha G = {\mathcal O} ( (h/\tilde h )^{-|\alpha|/2} \,, 
\ | \alpha| \geq 1 \,. \end{gathered}
\end{gather}
Using the rescaling \eqref{eq:resc}
we see that Proposition \ref{p:bbc} implies that 
\begin{gather}
\label{eq:imG}
\begin{gathered}
   \exp ( s G^w ( x, h D) ) = F_s^w ( x, h D) \,, \ \
F_s \in \ST ( m^s ) \,,  \ \ s \in \RR \,, \\
A = \Op ( a ) \,, \ \ a \in \ST ( m^s) \ \Longleftrightarrow 
\ A = \exp ( s G^w ( x, h D ) ) \Op ( a_0 ) \,, \ \ a_0 \in \ST \,. 
\end{gathered}
\end{gather}

For $ P \in \Psi^{0,0}_h ( X ) $ we consider
\begin{equation}
\label{eq:exGe}
\begin{split} 
P_{sG } \defeq e^{-s G^w ( x, h  D) / h } P e^{ G^w ( x, h D ) } & = e^{-\ad_{s G^w ( x, h D) } } P  
\in \PT \,, 
\end{split}
\end{equation}
where used Proposition \ref{p:bbc} as described above.
In particular we have an expansion
\begin{equation}
\label{eq:exGex} 
e^{-sG^w ( x, h  D) / h } P e^{ s G^w ( x, h D ) /h } \sim 
 \sum_{\ell=0}^\infty  \frac{(-1)^\ell} { \ell!} \left( 
s  \ad_{G^w ( x, h D ) }\right) ^\ell P \,, 
\end{equation}
where 
\begin{equation}
\label{eq:ads}
  \ad_{G^w ( x, h D ) } ^\ell P  \in h \tilde h^{\ell -1 } \PT \,.
\end{equation}

\subsection{Escape function away from the trapped set}
\label{eft}

Here we recall the escape function from 
\cite[Appendix]{GeSj}. Suppose that 
$ U, V $ are open neighbourhoods of 
of $ K \cap p^{-1} ( [ - \delta, \delta] )  $, 
\[   \overline U \Subset V  \Subset T^*X \,. \]
There exists $ G_1 \in \CI ( T^* X ) $, such that
\begin{equation}
\label{eq:gsa0}
G_1\rest_U\equiv 0\,, \  \ H_p G_1 \geq 0 \,, \ \
H_p G_1 \rest_{ p^{-1}( [-2\delta, 2 \delta ])   } \leq C \,,\ \ 
H_p G_1 \rest_{  p^{-1} ([- \delta, \delta ]  )\setminus V } \geq 1 \,.
\end{equation}
Since  $H_p G_1 \geq 0 $, $ G_1 $ is an escape function in the 
sense of \cite{HeSj}. It is 
{\em strictly increasing} along the flow of $ H_p $
on $ p^{-1}( [-\delta, \delta ] ) $, {\em away from} the trapped set $ K $.
Moreover $ H_p G $ is bounded in 
a neighbourhood of $ p^{-1} ( [-2 \delta , 2 \delta ] ) $. 
Such an escape function $G_1$ is necessarily of unbounded support.

\section{Proof of Theorem \ref{th1}}
\label{prth}

In \S \ref{ep}-\ref{en} we identify $ P ( z) $ with 
$ P_0 ( z ) $ and \emph{assume that $ u $ is supported in $ X_0 $.}
In \S \ref{glue} we will show how the assumptions on $ P(z) $
in \S \ref{ina} give a global estimate on the inverse. 
Since we have not assumed that $ \bbbone_0 ( P ( z) - i)^{-1} $ is
a compact operator we do {\em not} prove that $ ( P ( z ) - i W )^{-1} $
is a meromorphic family of operators. We prove that the inverse
exists for $ \Im z > - \nu_0 h $ by direct estimates.

\subsection{Estimates for $ \Im z > 0 $.}
\label{ep}

To obtain the first estimate in \eqref{eq:th1} we 
adapt the proof of \cite[Lemma 6.1]{SZ8} to our setting. 

For that let $ \psi^w  = \psi^w ( x , h D ) $, $ \psi 
\in \CIc ( T^*X , [0,1]) $,
be a microlocal cut-off to a
a small neighbourhood of $ p^{-1} ( 0 ) \cap T^*_{U_2 } X $, and 
suppose that 
$$ v = (  P ( z) - i W ) u \,. $$
Semi-classical elliptic regularity gives 
\begin{equation}
\label{eq:5.1.1}
\| ( 1 - \psi^w  ) u \| \leq C \| v \| + {\mathcal O} ( h^\infty ) 
\| u \|
\end{equation}
(see part {\em (i)} of Proposition \ref{p:work}).
The assumption that $ \partial_z P ( 0 ) $  has a negative
symbol on the characteristic set of $ p $, in the region where
$ 0 < W < 1 $ implies that
\[ P ( z ) = P ( \Re z ) - i \Im z \;  Q ( z ) \,, \]
where $ P ( \Re z ) $ is self-adjoint and $ \sigma (Q ( z ))  > 1/C > 0 $
near $ p^{-1} ( 0 ) \cap T^*_{U_2  } X $. This shows that 
\begin{equation}
\label{eq:5.1.2} 
\begin{split}
- \Im \langle ( P( z ) - i W ) \psi^w u , \psi^w u \rangle & =  \Im z \; \Re \langle 
Q ( z ) \psi^w u , \psi^w u \rangle + \langle W \psi^w u , \psi^w u \rangle \\
& \geq 
 \Im z \; \left( \| \psi^w u \|^2 / C - {\mathcal O}(h^\infty ) \| u \|^2  
\right) 
+  \langle W \psi^w u , \psi^w u \rangle 
\,, \end{split} \end{equation}
where we used the semi-classical G{\aa}rding inequality (see
\cite[Theorem 7.12]{DiSj} or \cite[Theorem 4.21]{EZB}).
We also write
\[ 
\begin{split}
\Im  \langle P ( z) u , u \rangle - \Im \langle P ( z) \psi^w u , \psi^w u 
\rangle &= \Im z  \left( \langle Q ( z ) u , u \rangle - 
\langle Q ( z ) \psi^w u , \psi^w u \rangle \right)\\ &= 
 \Im z \; {\mathcal O } ( 1) \| ( 1 - \psi^w ) u \| \| u \|\\  &= 
\Im z \; {\mathcal O} ( 1) \left( \| v \| \| u \| + {\mathcal O}
( h^\infty ) \| u \|^2 \right) \,, \end{split} \]
where we used elliptic regularity \eqref{eq:5.1.1} in the last 
estimate. Then, applying \eqref{eq:5.1.2},
\[ 
\begin{split}
\| u \| \|v \| & \geq  
- \Im \langle ( P ( z ) - i W ) u , u \rangle \\
& = 
- \Im \langle ( P ( z ) - i W ) \psi^w u , \psi^w u \rangle 
- \Im z \; {\mathcal O} ( 1) \left( \| v \| \| u \| 
+ {\mathcal O} ( h^\infty ) \|u\|^2 \right)  \\
& \ \ \ + \langle ( W  - \psi^w W \psi^w ) u , u \rangle  - {\mathcal O} ( h ) 
\| u \|^2  \\
& \geq \Im z  \left( \| \psi^w u \|^2/ C  - {\mathcal O}(1) \|v\| \| u \| 
- {\mathcal O} ( h ) \| u \|^2 \right) \,. \end{split} \]
Here  $ W - \psi^w W \psi^w \geq - {\mathcal O} ( h ) $ follows
from  the semi-classical sharp G{\aa}rding inequality. 

For small $ \Im z $ the term $ \|v\| \| u \| $ on the left hand side 
can be absorbed in the right hand side, and by adding 
$ \Im z \| (1 - \psi^w ) u \|^2 $ to both sides we obtain
\[ \Im z \| u \|^2 / C  \leq  \| u \| \| v \| + {\mathcal O} ( h ) 
\Im z \| u \|^2  \, , \]
and that gives
\[ \| u \| \leq \frac{C}{ \Im z }  \| v \|   \,. \] 
Combined with the estimates in \S \ref{glue}
this proves 
\[ \| (  P ( z ) - i W )  ^{-1} \|_{L^2 \rightarrow L^2} \lesssim
\frac 1 {\Im z } \,, \ \ \text{ for } \ \Im z > 0 \,, \ \ |z | < \delta_0 \,. \]

\subsection{Estimates on the real axis.}
\label{ere}

In this section we will use a commutator argument to obtain an
estimate on the real axis.  In fact, this bound automatically gives
holomorphy of $ ( P ( z ) - i W )^{-1} $ in $ \Im z > - \nu_1 h / \log
( 1/h ) .$ 

In this and the following sections we will assume that $ z = {\mathcal O} ( h ) $
so that we can work at a fixed energy level. That means that
\begin{equation}
\label{eq:PGz}
P ( z ) = P - z Q + {\mathcal O}_{H^2_h \rightarrow L^2 }  ( h^2 ) \,,
\end{equation}
$ P $ and $ Q $ are self-adjoint, and 
where $ Q = q^w ( x , h D) \in \Psi^{2,0}_h ( X ) $ is elliptic
and has a positive symbol in a neighbourhood of $ T^*_{ U_2 } X 
\cap p^{-1} ( [ - \delta, \delta ] ) $. The estimates are uniform 
when we shift the energy level within $ | \Re z | < \delta_0 $ and
hence we obtain the estimates in Theorem \ref{th1}.

For simplicity of the presentation we assume that $ \Gamma_\pm $ 
have global defining functions, that is that $ \Gamma_\pm $ are
orientable. The only object that needs to be globally defined, however, 
is the escape function $ G $ given in \eqref{eq:defG}. That 
involves only squares of defining functions, that is the 
$ d ( \bullet , \Gamma_\pm )^2 $, near $ K $, and these are 
well defined and smooth. 

We start with the following 
\begin{lem}
\label{l:defG}
Let $ \varphi_\pm $ be any defining functions of $ \Gamma_\pm $: 
\[ \Gamma_\pm = \{ \rho \; : \; \varphi_\pm ( \rho ) = 0 \} \,, \ \ \ 
d \varphi_\pm \rest_{\Gamma_\pm } \neq 0 \,. \]
Then, there exist $ c_\pm \in \CI ( T^* X ; \RR ) $ such that  
\begin{equation}
\label{eq:Hpp}
  H_p \varphi_\pm = \mp c_\pm^2 \varphi_\pm \,, \ \ \text{$ 
c_\pm > 0 $ in $ \neigh( K_0 ) $,} 
\end{equation}
and we can choose the sign of $ \varphi_\pm $ so that 
\begin{equation}
\label{eq:pois}  \{ \varphi_+ , \varphi_- \} \rest_K > c_0 > 0  \,.
\end{equation}
\end{lem}
\begin{proof}
Since $ H_p $ is tangent to $ \Gamma_\pm $ we have $
H_p \varphi_\pm = \alpha_\pm \varphi_\pm $ and 
$ H_p \varphi_\pm^2 = 2 \alpha_{\pm } \varphi_{\pm}^2 $. 
To see that $ \mp \alpha_\pm > 0 $, we need to check that
\[  H_p d( \bullet , \Gamma_\pm)^2 = 
\frac d {dt} \exp ( t H_p) ^* d ( \bullet , \Gamma_\pm )^2  \big|_{t=0} 
\sim \mp d ( \bullet, \Gamma_\pm )^2 \,, \ \ 
\text{ in } \neigh(K_0) \,. \]
But this follows from the assumption \eqref{eq:normh} which implies that
\[  d( \exp ( \pm t H_p ) (\rho) , \Gamma_\pm )^2 \leq C \exp( - \theta t ) d
( \rho, \Gamma_\pm )^2 \,, \ \ 0 \leq t \leq T  \,, \]
for $ \rho$ in a $T$-dependent neighbourhood of $ K_0$ 
-- see \cite[Lemma 5.2]{SjDuke}.

To see \eqref{eq:pois} we note that $ d\varphi_\pm ( \rho) $,  
$ \rho \in K_0$,  are linearly independent and vanish on 
$ T_\rho K_0 \subset T_\rho ( T^*X ) $ which is a symplectic manifold
of codimension $ 2 $. Hence $ (H_{\varphi_\pm})_\rho  $ are linearly independent
and transversal to $ T_\rho K_0 $, and 
\[ \{ \varphi_- , \varphi_+ \} ( \rho ) = \omega_\rho ( H_{\varphi_+ } , 
H_{\varphi_-} ) \neq 0 \,, \]
because of the non-degeneracy of $ \omega $, the symplectic form. 
 If necessary switching the sign of one
of the $ \varphi_\pm$ we can then obtain \eqref{eq:pois}.
\end{proof}

We define
\begin{equation}
\label{eq:defG}
G ( \rho ) =  \chi ( \rho ) \log 
   \frac{ \varphi_-^2 ( \rho ) + h/\tih } 
{ \varphi_+^2 ( \rho) + h/\tih }  + C_1 \log \left( \frac 1 h \right)
\chi_1 ( \rho ) G_1 ( \rho)    \,, 
\end{equation}
where: $ \chi \in \CIc ( T^* X ) $ is supported near $ K_0 $, with
$\chi=1$ on the set $V$ in \eqref{eq:gsa0}; $ G_1 $ 
is described in \S \ref{eft};
$ \chi_1 \in \CIc ( T^*X ) $, 
\[ \chi_1 ( \rho ) \equiv 1 \,,  \ \  \rho \in 
 p^{-1} ( [ -2 \delta, 2 \delta ] ) \cap T_{ B( 0 , 2 R ) } ^* X \,; \]
 $\supp \nabla \chi
\subset \{\chi_1=1\};$
and $ C_1 $ is a large constant. Writing $ G^w = G^w ( x, h D) $ 
we observe that 
\begin{equation}
\label{eq:logh}
   \| G^w u \|_{H^k_h } \leq \log ( 1/h ) \| u \|_{ L^2 } \,, \ \ \forall \, k 
\,.
\end{equation}
We also recall an elliptic estimate:
\begin{equation}
\label{eq:elli} \| ( P - i W ) u \|_{L^2}  \geq \| ( 1 - \psi_1^w )(  P - i W) 
u \| \geq  
\frac 1 C \| ( 1 - \psi_2^w ) u \|_{ H^2_h } - {\mathcal O} ( h^\infty ) \| u 
\|_{L^2 } \,, \end{equation}
where $ \psi_j \in \CIb (T^*X ; [ 0 , 1 ] ) $ are as in 
Proposition \ref{p:work}.
In fact, if $ E_1 $ has the properties given in 
that proposition, 
\[ \begin{split} \| ( 1 - \psi_2^w ) u \|_{H^2_h}  &  = \| ( 1 - \psi_2^w ) E_1 
(  P - i W ) u \|_{H^2_h} 
  + {\mathcal O}(h^\infty ) \| u \| \\
& = 
\| ( 1 - \psi_2^w ) E_1 ( 1 - \psi_1^w ) ( P - i W ) u \|_{H^2_h} 
  + {\mathcal O}(h^\infty) \| u \| \\
& \leq C \| ( 1 - \psi_1^w ) ( P - i W ) u \|_{L^2} 
  + {\mathcal O}(h^\infty) \| u \| \,, \end{split} \]
which is \eqref{eq:elli}.

The elliptic estimate shows that we only need to prove
\[  \| ( P ( z ) - i W ) u \| \geq \frac{ h} {\log ( 1/h ) } \,, \]
for $ u $ satisfying 
\begin{equation}
\label{eq:loch} \widetilde\chi^w ( x, h D ) u = u + {\mathcal O}_{H^k_h} ( h^\infty ) \,, \ \
\| u \| = 1 \,, \end{equation}
where $ \widetilde\chi $ has properties of, say, $ \psi_1 $  in \eqref{eq:elli}.
That is because the commutator 
terms appearing after this localization can be estimated using \eqref{eq:elli}. 

Hence from now on we assume that $ u $ satisfies \eqref{eq:loch}
with the support of $ \chi $ in a small neighbourhood of the 
energy surface $ p^{-1} ( 0 ) $.

We now proceed with the positive commutator estimate.
Let $ M_0> 0  $, $ \RR \ni z = {\mathcal O} ( h ) $, and calculate
\begin{equation}
\label{eq:posc} 
\begin{split} 
& - 2\Im \langle ( P(z) -  i W ) u , ( G^w + M_0 \log ( 1/h)) u \rangle \\
& \ \ \ = 
- 2 \Im \langle ( P - z Q  -  i W ) u , ( G^w + M_0 \log ( 1/h)) u \rangle 
- {\mathcal O} ( h^2 ) \| u \|^2 \\
& \  \ \  =  - i \langle [ P , G^w  ] u ,  u \rangle +
2 M_0 \log (1/h)  \langle W  u ,  u  \rangle +  
2 \langle W u , G^w u \rangle   - {\mathcal O} ( h^2 ) \| u \|^2   \\
&  \ \ \ \geq h  \langle ( H_p G )^w u , u \rangle + 
2 M_0 \log (1/h)  \langle W  u ,  u  \rangle
 - 2 \| W u \| \| G^w  u \| - {\mathcal O} ( h^{\frac 32} 
\tih ^{\frac 32} ) \| u \|^2 \\
& \ \ \ \geq   h  \langle ( H_p G )^w  u , u \rangle
+  {M_0}  \log (1/h) \| W^{\frac 12} u \|^2 - 
{\mathcal O} ( h^{\frac 32} 
\tih ^{\frac 32} ) \| u \|^2 
\,, 
\end{split}
\end{equation}
where we used the fact that $ 0 \leq W \leq \sqrt W $ and 
chose $ M_0 $ large enough.

To analyze $ (H_p G )^w ( x, h D ) $ we proceed locally 
using the invariance properties described in \S \ref{ptm}:
the resulting errors are of lower order. To keep the 
notation simple we write the argument as if $ \varphi_\pm $
were defined globally (which is the case when 
$ \Gamma_\pm $ are orientable).

The crucial calculation is based on Lemma \ref{l:defG}:
\begin{equation*}
 H_p G = \frac{ ( c_+\varphi_+ )^2 } { \varphi_+^2 + h / \tih } 
+ \frac{ ( c_-\varphi_- )^2 } { \varphi_-^2 + h / \tih } +R_0+R_1\in 
\tilde S_{\frac12} \,, \ \ \text{ in $ \neigh(K_0)  $;}
\end{equation*}
here $R_0$ is the term arising from $H_p(\chi)$ and $R_1$ from
$H_p(\chi_1).$

Put 
\[ \Phi_\pm \defeq  \widehat \varphi_\pm ^w ( x, h D ) 
\in \PT 
\,, \ \
\widehat \varphi_\pm \defeq 
 \frac{ c_\pm\varphi_\pm  } { \sqrt{\varphi_\pm^2 + h / \tih} } \,.
\]
We now recall the properties of $ G_1 $ enumerated in \S
\ref{eft}; note further that $\supp R_0\subset \{H_pG_1\geq 1\},$ hence for
$C_1 \gg 0$ we may absorb the $R_0$ term into the term arising from
$H_p G_1,$ and obtain the following global description of $H_p G:$
\begin{equation}
\label{eq:cru1} H_p G = \widehat \varphi_+^2 + \widehat \varphi_-^2 +R_1+ C_1 \log( 1/h ) a\,,
\end{equation}
where $ a \in S ( T^* X ) $, and 
\[  a ( \rho ) \geq 1/2 \,, \ \  d ( \rho, K) > \epsilon > 0 \,, \ \ 
\rho \in \neigh ( p^{-1} ( 0 ) )  \,, \ \ 
\rho \in U_2 \,. \] 

We should now remember that using the rescaling \eqref{eq:resc}
we are now in the semiclassical calculus with the $ \tilde h $ 
Planck constant. That means that the Weyl quantization is equivalent
to the $ \tih $ quantization.

Then \eqref{eq:cru1} and the fact that we are using the Weyl
quantization show that
\[  ( H_p G )^w ( x, h D ) = \Phi_+^2 + \Phi_-^2 + C_1 \log(1/h) a^w ( x, h D) 
+R_1^w + {\mathcal O}_\PT ( \tih^2 ) \,. \] 

We now write
\[ \Phi_+^2 + \Phi_-^2 = \Phi^* \Phi 
+ i [ \Phi_+ , \Phi_-] \,, \ \ 
\Phi \defeq \Phi_+ - i \Phi_-\]
so that, without writing the terms involving $ a^w ( x, h D) $ and $R_1^w$,
\begin{equation}
\label{eq:lowG}
\begin{split}
\langle ( H_p G )^w ( x, h D)  u , u \rangle & \geq
\langle  ( \Phi_+^2 + \Phi_-^2 ) u , u \rangle - {\mathcal O} ( \tih^2 ) 
\| u \|^2  
\\
& \geq  \|  \Phi u \|^2_2 + \langle 
i [ \Phi_+ , \Phi_-] u , u \rangle - {\mathcal O} ( \tih^2 ) \| u \|^2  
\\
& \geq M \tih  \|  \Phi  u \|^2_2 + 
h \langle \{ \widehat \varphi_+ , \widehat  \varphi_- \}^w ( 
x, h D) u , u \rangle - {\mathcal O} ( \tih^2 ) \| u \|^2 
\\
& \geq \langle ( M \tih ( \widehat \varphi_+^2 + \widehat \varphi_-^2 ) + 
h \{ \widehat \varphi_+ , \widehat \varphi_- \} )^w ( x, h D ) u , u 
\rangle - {\mathcal O} ( \tih^2 ) \| u \|^2 
\,,
\end{split} \end{equation}
where $ M $ is some large constant.
Putting 
$$ \widetilde \varphi_\pm \defeq (\tilde h/h )^{\frac12} \varphi_\pm \,,
$$ 
we calculate
\[ \begin{split}   h \{ \widehat \varphi_+, \widehat \varphi_- \}  = &
\frac { \tih c_+ c_-  \{ \varphi_+ , \varphi_- \} }
{  ( 1 + \tilde \varphi_+^2 )^{\frac 32 } 
 (  1 + \tilde \varphi_-^2 )^{\frac 32 } } +
\frac {  ( h \tih )^{\frac 12} \tilde \varphi_+ \{ c_+ , \varphi_- \} }
{ ( 1  + \tilde \varphi_+^2 )^{\frac 12 } 
 ( 1   + \tilde \varphi_-^2 )^{\frac 32 } }  \\ 
& \ \ + 
 \frac {   ( h \tih )^{\frac 12}  \tilde \varphi_+  \{ c_- , \varphi_+ \} }
{ ( 1  + \tilde \varphi_+^2 )^{\frac 32 } 
 ( 1  + \tilde \varphi_-^2 )^{\frac 12 } } + 
\frac { h \tilde \varphi_+ \tilde \varphi_-  \{ c_+ , c_- \} }
{ ( 1  + \tilde \varphi_+^2 )^{\frac 32 } 
 ( 1  + \tilde \varphi_-^2 )^{\frac 32 } } \\
& = \frac { \tih c_+ c_-  \{ \varphi_+ , \varphi_- \} }
{ ( 1  + \tilde \varphi_+^2 )^{\frac 32 } 
 ( 1  + \tilde \varphi_-^2 )^{\frac 32 } }  - {\mathcal O}_\ST ( ( h \tih )^{\frac 12} 
)
\end{split}
\]
Hence 
\[ \tilde \varphi \defeq  M \tih ( \widehat \varphi_+^2 + \widehat \varphi_-^2 ) + 
h \{ \widehat \varphi_+ , \widehat \varphi_- \}  \, \]
satisfies 
\[ \tilde \varphi \in \tih \ST \,, \]
and, using \eqref{eq:pois}, we obtain near $ K_0$, 
\[ \begin{split}
\tilde \varphi & = 
 \tih \left( M ( \tilde \varphi_+^2 + \tilde \varphi_-^2 ) + 
 \frac { c_+ c_-  \{ \varphi_+ , \varphi_- \} }
{ ( 1  + \tilde \varphi_+^2 )^{\frac 32 } 
 ( 1  + \tilde \varphi_-^2 )^{\frac 32 } }  
- {\mathcal O}_\ST ( (  h /{\tih} )^{\frac 12} ) \right)  \\
& \geq 
\tih \left( M ( \tilde \varphi_+^2 + \tilde \varphi_-^2 ) + 
 \frac { c_0} { ( 1  + \tilde \varphi_+^2 )^{\frac 32 } 
 ( 1  + \tilde \varphi_-^2 )^{\frac 32 } }  
- {\mathcal O}_\ST  ( (  h /{\tih} )^{\frac 12} ) \right) \\ 
 & 
\geq c_1 \tih \,, \ \ \ c_1 > 0 \,. 
\end{split} \]

We now return to \eqref{eq:posc} which combined with 
\eqref{eq:logh},\eqref{eq:cru1}, and the above definition of $ \tilde \varphi $
gives, for some large constant $ M_1 $, and $ \RR \ni z$,
$ u $ satisfying \eqref{eq:loch},
\[ \begin{split}  M_1 \log ( 1/h)  \| ( P ( z ) - i W ) u \| \| u \| 
& \geq \langle ( h \tilde \varphi^w  + hR_1^w + C_1 \log(1/h) a^w + M_0 \log ( 1/h ) W ) u , u 
\rangle  \\
& \geq h \langle ( \tilde \varphi^w+ R_1^w +  \log ( 1/h ) b^w ) u , u \rangle 
\end{split} \]
where, as $ W \geq 0 $,  
\[ b \defeq  C_1 a + M_0 W \geq 0 \ \  \Longrightarrow \ \ 
b^w ( x, h D ) \geq - C h \,, \] 
with the implication due to the sharp G{\aa}rding inequality. 
We also observe that 
\[ \tih \ST \ni \tilde \varphi + \tilde h b \geq c_1 \tih, \ \ c_1 > 0 \,, \]
near $ p^{-1} ( ( - \delta , \delta ) ) $.
Furthermore, since $u$ is assumed to satisfy
\eqref{eq:loch}, and as we have $R_1^w=\mathcal{O}(h^\infty)$ on such
distributions,
we obtain
\[ \begin{split}  M_1 \log ( 1/h)  \| ( P ( z ) - i W ) u \| \| u \| 
& \geq h \langle ( \tilde \varphi^w +  \tih b^w ) u , u \rangle 
- {\mathcal O} ( h^2 \log(1/h) ) \| u \|^2 \\
& \geq c_3 \tih h \|u \|^2 \,, \ \ c_3> 0 \,, 
\end{split} \]
which proves the bound \eqref{eq:th1} for $ \Im z = 0 $.

\subsection{Estimates for $ \Im  z > - \nu_0 h $}
\label{en}

To prove the estimates deeper in the complex plane we 
will use exponentially weighted estimates which use the
same escape function $ G $ given in \eqref{eq:defG}. 
We start with a lemma 
which is based on \cite[Proposition 7.4]{SZ10}:
\begin{lem}
\label{l:defm}
Let $ G $ be given by \eqref{eq:defG} above. Then for 
$ \rho , \rho' $ in any compact neighbourhood of $ K_0 $ we have 
\[  \frac { \exp G  ( \rho ) } { \exp G ( \rho' ) } 
 \leq C \left 
\langle \frac {\rho - \rho' }
{(h/\tih)^{\frac12} } \right \rangle^N \,, \ \ N > 0 \,. \] 
In particular,
\[ m ( \rho ) \defeq 
\exp G ( \rho ) \] 
is an order function for the $ \PT $ calculus, that is, satisfies
\eqref{eq:ord1}.
\end{lem}
\begin{proof}
For the reader's convenience we recall the slightly modified
argument.  We first claim that
\begin{equation}
\label{eq:claim1}
  \frac{ \varphi_{\pm } ( \rho)^2 + h/\tih } {\varphi_\pm ( \rho')^2
+ h/\tih } \leq 
C_1 \left \langle
\frac{ \rho - \rho'}{(h/\tih)^{\frac12}  } \right \rangle^2 \,:
\end{equation}
Since $ \varphi_\pm^2 \sim d ( \bullet , \Gamma_\pm )^2 $, 
we have 
\[ \begin{split} 
\varphi_\pm ( \rho )^2  + h/\tih & \leq C ( d ( \rho , \Gamma_\pm )^2 + 
h/\tih ) 
\leq C ( d ( \rho', \Gamma_\pm)^2  + | \rho'- \rho|^2 + 
h/\tih ) \\
& \leq C' ( \varphi_\pm ( \rho' )^2 + h/\tih  + | \rho' - \rho|^2 
) \\
& = C' ( \varphi_\pm ( \rho' )^2 + h/\tih  
 + (h/\tih)  \langle ( \rho - \rho' )/ ( h/\tih)^{\frac12} \rangle^2 ) \\
& \leq 2 C'  ( \varphi_\pm ( \rho' )   ^2 + h/\tih )
  \langle ( \rho - \rho' )/ (h/\tih)^2 \rangle^2 \,.
\end{split} \]
which proves \eqref{eq:claim1}. In other words, for
\[ \widehat G ( \rho ) \defeq 
 \log 
   \frac{ \varphi_-^2 ( \rho ) + \epsilon^2 } 
{ \varphi_+^2 ( \rho) + \epsilon^2 } \,, \ \  \epsilon = 
\left(\frac h {\tih} \right)^{\frac12} \,, \] 
we have
\[ | \widehat G ( \rho ) - \widehat G ( \rho' ) | \leq C + 2  \log 
\langle ( \rho - \rho' ) / \epsilon \rangle \,.\]
For $  \chi \in \CIc $, 
\[ |  \chi ( \rho ) \widehat G ( \rho ) - 
 \chi ( \rho' ) \widehat G ( \rho' ) | \leq 
C | \rho - \rho' | \log ( 1 / \epsilon ) + C \log \langle ( \rho - \rho' ) /
 \epsilon \rangle \,. \]
Also, 
\[ | \chi_1 ( \rho ) G_1 ( \rho ) - \chi_1 ( \rho' )  G_1 ( \rho' ) | 
\leq C | \rho - \rho' | \log ( 1/ \epsilon ) \,, \]
with $G_1$ as in \S\ref{eft};
thus to prove the lemma
we need 
\[ | \rho - \rho' | \log \frac 1 \epsilon 
 \leq C  \log \langle ( \rho - \rho' ) /
\epsilon  \rangle 
+ C  \,, \  \ \rho, \rho' \in Q \Subset \RR^{2n} \,.\]
If we put $ t = | \rho - \rho' |/ ( C\epsilon) $, this becomes
\[  \epsilon \log \frac{1}{\epsilon } 
\leq \frac{\log \langle t \rangle + 1  }t   \,,  \ \ 0 < t \leq 
\frac1{\epsilon} \,, \]
which is acceptable as the function $ t \mapsto ( \log  \langle t \rangle  + 1) / t $
is decreasing.
\end{proof}

We now consider $ ( P( z ) - i W ) _{s G} $ defined by 
\eqref{eq:exGe} using this weight function $ G $. 
Then using \eqref{eq:ads} and Lemma \ref{l:err} (to understand 
$ \ad_{sG^w}  P  $), 
\[ \begin{split}
& P(z)_{sG}  = P - i s h ( H_p G)^w ( x, h D ) -
 z Q  + {\mathcal O}_{\PT} ( 
s^2 \tih h + s h^{\frac32} \tih^{\frac32} + h^2 ) \,. 
\end{split} \]
and
\[ W_{sG } = W + i s h \log(1/h) ( H_{\rho_1 G_1} W )^w ( x, h D )  
 + {\mathcal O}_{\Psi_h} ( s^2 (h \log(1/h) )^2)    \,, \]
where $ G_1 $ and  $ \rho_1 $ are as in \eqref{eq:defG}.
Hence,
\[ \begin{split} 
- \Im \langle ( P( z ) - i W )_{sG} u , u \rangle 
& = \langle (  H_p G )^w + W - \Im z \, Q ) u , u \rangle \\
& \ \ \ \ \ +  {\mathcal O}_{\PT} ( 
s^2 \tih h + s h^{\frac32} \tih^{\frac32} + h^2 ) \,.
\end{split} \]
For $ u $ satisfying \eqref{eq:loch}, $ s> 0  $ small, $ \Im z > - \nu_0 h $ for 
a sufficiently small $ \nu_0 $, we can now proceed as at the end of 
\S \ref{ere} to obtain invertibility:
\[ c_1 h \tilde h  \| u \| \leq \|  ( P ( z ) - i W )_{sG} u \| \,, \ 
\ \ \Im z > - c_0 h \tih \,, \ \ | z | \leq C h \,. \]
Since 
\[ \exp (\pm s G^w ( x, h D) )  = \mathcal{O}_{L^2 \rightarrow L^2} 
 ( h^{-k } ) \,, \]
that means that
\[ h^{k_1}   \| u \| \leq C_1 \|  ( P ( z ) - i W )  u \| \,, \ 
\ \ \Im z > - c_0 h \tih \,, \ \ | z | \leq C h \,. \]

\subsection{A global estimate}
\label{glue}
Here we show how the assumption \eqref{eq:nasty} 
part {\em (ii)} of Proposition \ref{p:work} give 
a global estimate; recall that the estimates of \S\ref{ep}--\ref{en}
applied to $u$ supported in $X_0.$  We fix a partition of unity on the
interior of $X_0$
$$
1=\chi_0^2+\chi_1^2 
$$
such that $\chi_0=1$ on $U_2,$ $\supp \chi_1 \subset
\{W=1\},$ and with $\supp \chi_i\subset \{W >0\}$ for $i=1,2.$

The results of \S \ref{ep},\ref{ere}, and \ref{en} show that, in the
notation of \S \ref{ina}, 
\begin{equation}
\label{eq:eW} \gamma ( z , h )  \| \chi_0 u \| 
\leq C \| ( P_0 ( z ) - i W ) \chi_0 u \|   \,, \ \ 
\gamma ( z, h ) \defeq 
 \begin{cases} 
\Im z \,,  & \Im z > 0  \,, \\
h/ \log (1/h) \,, &   \Im z=0\,, \\ 
h^{k} \,,  &    \Im z > -\nu_0 h  \,,
\end{cases} \end{equation}
and, since $ \chi_1 W = 1 $, 
\begin{equation}
\label{eq:ei}  c_0 \| \chi_1 u \| \leq \| ( P ( z ) - i W ) \chi_1 u \| \,,
\end{equation}
as implied by the hypothesis \eqref{eq:nasty}. 

Now, writing $ \widetilde P ( z ) = P ( z ) - i W $, 
\begin{equation*}
\begin{split}
 \| \widetilde P ( z )  u \|^2 &  = \| \chi_0 \widetilde P ( z )  u \|^2 + \| \chi_1 \widetilde P ( z )  u \|^2  \\
&  \geq  \| \widetilde P ( z )  \chi_0  u \|^2 + \| \widetilde P ( z )  \chi_1  u \|^2 
-  \| [ \chi_0, \widetilde P ( z )  ] u \|^2  -  \| [ \chi_1, \widetilde P ( z )  ] u \|^2 \\
& \hspace{0.2in}  - \; 2  \left( \| \chi_0 \widetilde P ( z )  u \|\| [ \chi_0, \widetilde P ( z )  ] u \|
+ \| \chi_1 \widetilde P ( z )  u \|
\| [ \chi_1, \widetilde P ( z )  ] u \|
\right)
 \\
&  \geq  \| \widetilde P ( z )  \chi_0  u \|^2 + \| \widetilde P ( z )  \chi_1  u \|^2 
  - 2 C ( \| [ \chi_0, \widetilde P ( z )  ] u \|^2  +
 \| [ \chi_1, \widetilde P ( z )  ] u \|^2 )
\\
& \hspace{0.2in}
-  \| \widetilde P ( z )   u \|^2/C   
\end{split} \end{equation*}
Since on the support of the commutator terms $ W = 1 $ and 
$ P ( z) = P_0 ( z ) $, we have
obtained 
\[ \begin{split}
C_0 \| ( P ( z ) - i W ) u \|^2 & \geq \| ( P_0 ( z ) - i W ) \chi_0 u \|^2 
+ \| ( P ( z ) - i ) \chi_1 u \|^2 \\ 
& \hspace{0.2in} - C_1 
( \| [ \chi_0 ,  ( P_0 ( z ) - i )  ] u \|^2  +
 \| [ \chi_1 , (  P_0 ( z ) - i )   ] u \|^2 ) \,. 
\end{split} \]
Using {\em(ii)} of Proposition \ref{p:work} we obtain
\[ \begin{split} \| [ \chi_j,  ( P ( z ) - i )  ] u \|^2 & \leq 
C h^2  \|  \psi ( P_0  ( z ) - i ) u \|^2 - 
{\mathcal O} ( h^\infty ) \| u \|_2^2 \\
& \leq C h^2 \| ( P ( z ) - i W ) u \|^2
- 
{\mathcal O} ( h^\infty ) \| u \|_2^2 \,,  
 \end{split} \]
where $ \psi \in \CIc ( X_0) $ satisfies 
$$ W \rest_{ \supp \psi } 
\equiv 1 \,, \ \ \ \psi \rest_{\supp d\chi_j } \equiv 1 \,. $$
We apply this estimate, \eqref{eq:eW}, and \eqref{eq:ei}, to get
\[ \begin{split}
C_2 \| ( P ( z ) - i W ) u \|^2 & \geq \| ( P_0 ( z ) - i W ) \chi_0 u \|^2 
+ \| ( P ( z ) - i ) \chi_1 u \|^2 - {\mathcal O}(h^\infty ) \| u \|^2 
\\
& \geq \gamma ( z, h ) \| \chi_0 u \|^2 + 
c_0 \| \chi_1 u \|^2 - {\mathcal O} ( h^\infty ) \| u \|^2 \\
& \geq \gamma ( z , h ) ( \| \chi_0 u \|^2 + \| \chi_1 u \|^2 ) 
- {\mathcal O} ( h^\infty ) \| u \|^2 \\
& \geq (\gamma ( z , h ) /2) \| u \|^2 \,. 
\end{split} \]
which 
completes the proof of Theorem~\ref{th1}.

\section{Results for resonances}
\label{rere}
Here we briefly indicate how the proof presented in 
\S \ref{prth} adapts to give a resonance free strip.
First we need to make additional assumptions on the
operator guaranteeing meromorphic continuation of the
resolvent.

Suppose that $ X $ is given by \eqref{eq:X} with 
$ N \geq 1 $.
For simplicity
we will assume that $ N = 1 $, with obvious modifications
required when for $ N > 1 $. 

\renewcommand\thefootnote{\ddag}%

We make the same assumptions\footnote{We assume that 
$ p_1 $ is of order $ 1 $ in $ \xi $ to make the 
case of $ h = 1 $ easier to state.} as in \cite[(1.5)-(1.6)]{SZ10}
and  \cite[\S 3.2]{NZ3}: $ P = P ( h ) = P ( h )^* $, 
\begin{gather}
\label{eq:gac}
\begin{gathered}
 P(h) = p^w ( x, hD) + h p_1^w ( x , h D; h ) \,, \ \ 
p_1 \in S^{1,0} ( T^* X) 
\,, \\
 | \xi | \geq C \ \Longrightarrow \ 
p ( x , \xi ) \geq \langle \xi \rangle^2 / C \,, 
\ \ \ p = E \ \Longrightarrow dp \neq 0 \,,
\\
\exists \; R_0,  \ \forall \;  
u \in \CI ( X \setminus X_0 )\,, \ \  P ( h ) u ( x ) = 
P_\infty ( h ) u ( x ) \,, 
\end{gathered}
\end{gather}
where in $ X \setminus X_0 = \RR^n\setminus B(0,R) $ 
\begin{equation}
\label{eq:Q}
 P_\infty(h)  =\sum_{|\alpha|\leq2} a_{\alpha}(x;h){(hD_x)}^{\alpha}    \,,
\end{equation}
with 
$a_\alpha(x;h)=a_\alpha(x)$ independent of $h$ for $|\alpha|=2$,
$a_{\alpha}(x;h)\in C_b^\infty(\RR^n)$ uniformly bounded with respect to
$h$ (here $C_b^\infty(\RR^n)$
denotes the space of $C^\infty$ functions with bounded derivatives 
of all orders), and 
\begin{gather}
\label{eq:valid}
\begin{gathered}
  \sum_{|\alpha|=2} a_\alpha(x) {\xi}^\alpha\geq {(1/c)}{|\xi|}^2, \;\; 
{\forall\xi\in\RR^n}\,, \text{ for some constant $c>0$, } \\
\sum_{|\alpha|\leq2}a_\alpha(x;h){\xi}^\alpha \, \longrightarrow\, 
{\xi}^2 \,,\quad  \text{as $ |x|\to \infty$, uniformly with respect to $h$.}
\end{gathered}
\end{gather}
We further take the {\em dilation analyticity} assumption to hold
in a neighbourhood of infinity:
there exist $\theta_0\in {[0,\pi)},$ $\epsilon>0$ such that the
coefficients $a_\alpha(x;h)$ of $P_\infty(h)$ extend holomorphically in $x$ to
$$\left\{ r\omega: \omega\in {{\mathbb C}^n}\,, \ \ 
 \text{dist}(\omega,{\bf{S}}^n)<\epsilon\,, \ 
r\in {{\mathbb C}}\,, \   |r|>R_0\,, \ 
 \text{arg}\,r\in [-\epsilon,\theta_0+\epsilon) \right\} \,, $$
with \eqref{eq:valid} valid also in this larger set of $x$'s.

We note that more general assumptions are possible. We could
assume that $ X $ is a scattering manifold which is analytic 
near infinity and satisfies the conditions introduced in 
\cite{WZ}.

\begin{thm}
\label{th2}
Suppose $ P $ is an operator satisfying the dilation 
analyticity assumptions above and such that $ P( z ) = P - z $ 
satisfies the assumptions of Theorem \ref{th1}. 
Then for any $ \chi \in \CIc ( X ) $, $ \chi ( P - z )^{-1 } \chi $,
continues analytically from $ \Im z > 0 $ to $ \Im z > - \nu_0 h $,
$ | z | < \delta_0 $, and 
\begin{equation}
\label{eq:th2} \| \chi ( P - z  ) ^{ -1 } \chi \|_{L^2 \rightarrow L^2 }
\    \leq \ \begin{cases} 
C_\chi h^{-1} \log (1/h)\,, &   \Im z=0\,,  \\ 
C_\chi h^{-k} \,,  &    \Im z > -\nu_0 h \,,
\end{cases}\end{equation}
for $ |z | < \delta_0 $. 
In other words, there are no resonances in a strip of width proportional
to $ h $.
\end{thm}
\medskip

\noindent
{\bf Sketch of the proof:}
The proof follows the same strategy as the proof of 
the estimate $ {\mathcal O} ( h^{-k} ) $ for $ \Im z > - \nu_0 h $
in Theorem \ref{th1} but with $ W $ replaced by complex scaling
with angle $ \theta \sim h \log ( 1/h) $. That requires a finer
version of Lemma \ref{l:defm} which is given in 
\cite[Proposition 7.4]{SZ10}. In particular, the choice of the
cut-off function $ \chi_1 $ has to be coordinated with complex scaling 
(see also \cite[\S 4.2]{SZ10}). The same exponential weight
can then be used, following the arguments of \cite[\S 8.4]{SZ10}, but 
without the complications due to second microlocalization
needed there. 

This provides the bound $ {\mathcal O} (h^{-k})  $ for the 
norm of the analytically continued cut-off 
resolvent, $ \chi ( P - z)^{-1} \chi $, for $ \Im z > - \nu_0 h $.
To obtain the bound on the real axis we can proceed either as in 
\S \ref{ere}, or using the ``semiclassical maximum principle''
-- see for instance \cite[Lemma 4.7]{Bu} or \cite[Lemma A.2]{BZ2}.
\stopthm

Ideas used in the semi-classical case provide 
results in the case of the classical wave equation.
We first note that if $ P = P ( 1 ) $ satisfies the assumptions 
above then the resonances are defined as poles of the meromorphic 
continuation of $ ( P - \lambda^2)^{-1} $ from $ \Im \lambda > 0 $
to $ \Im \lambda > - c_0 | \Re \lambda | $ -- 
see \cite{Sj}. When $ P_\infty = -\Delta $ and the dimension, $ n $, 
is odd, the meromorphic continuation extends to the entire complex
plane (that is why we use the parametrization $ 
z = \lambda^2$, and when $ n $ is even we pass to the infinitely sheeted logarithmic
plane) -- see \cite{SZ1}. Theorem~\ref{th2} implies that
for $  \chi \in \CIc( X ) $,
\begin{equation}
\label{eq:la2}  \| \chi ( P - \lambda^2 )^{-1} \chi \|_{ L^2 \rightarrow L^2} 
\leq C_\chi |\lambda|^k \,, 
 \ \  \Im \lambda > - \alpha_1 \,, \ | \Re \lambda | >  \alpha_0 \,, \ \ \alpha, \beta > 0 \,.
\end{equation}
To relate this to energy decay we procceed in the spirit of \cite{BZ1}. 
Suppose that the operator $P$ satisfies the assumptions above with  $h=1$
and consider the wave equation for $ P $ with compactly supported
initial data:
\begin{equation}
\label{eq:waveP}
( D_t ^2 - P ) u = 0 \,, \ \ u \rest_{t=0} = u_0 \,, \ \
D_t u \rest_{ t = 0 } = u_1 \,, \ \ \supp u_j \subset V \Subset X \,. 
\end{equation}
The local energy decay results are 
different depending on finer assumptions on $P$ which
we state as three cases:

\begin{center}
\begin{tabular}{||l|l|l||} \hline
\ & \ & \\
\ Case 1 \ & \ \ $ P |_{ {\RR^n \setminus B ( 0 , R_0 ) } } = - \Delta 
 |_{ {\RR^n \setminus B ( 0 , R_0 ) } }$ \ \  & \ $ n $ odd \ \\ 
\ & \ & \\
\hline
\ & \ & \\
\ Case 2 \ & \ \ $ P |_{ {\RR^n \setminus B ( 0 , R_0 ) } } = - \Delta 
 |_{ {\RR^n \setminus B ( 0 , R_0 ) } }$ \ \ & \ $ n $ even \ \\ 
\ & \ & \\
\hline
\ & \ & \\
\ Case 3 \ & \ \ $ P |_{ {\RR^n \setminus B ( 0 , R_0 ) } } = P_\infty  
 |_{ {\RR^n \setminus B ( 0 , R_0 ) } }$ \ \ & \ any $ n $ \  \\ 
\ & \ & \\ \hline
\end{tabular}
\end{center}
\medskip
\noindent 
where $ P_\infty $ is an elliptic operator close to the Laplacian at infinity
-- see \eqref{eq:Q} and \eqref{eq:valid} --  with $ h= 1 $. 
\begin{thm}\label{t:dec}Let $ P $ be an operator satisfying the 
assumptions above with  $h=1$.
Let $ U , V \subset X $ be bounded open sets, and let 
$\Psi \in C^{\infty}\left( \mathbb{R}\right)$ be an even function such that
 \begin{equation}
 \Psi\left(x\right)=1 \ \ \begin{cases} \text{for  $x \in \RR$ in cases $1$ and $2$}\\
 \text{for $x \geq 1$ in case $3$} \end{cases}, \qquad
\Psi\left( x\right) = 0 \ \ \text{ near $0$ in case $3$.}
\end{equation}
Suppose that $ P $ has neither discrete spectrum nor a resonance
at $ 0 $. Then 
there  exists $ K > 0 $ such that the solutions of \eqref{eq:waveP} 
with 
\[  \| u_0 \|_{H^{K + 1}} \leq 1 \,, \ \ \| u_1 \|_{ H^K } \leq 1 \,, 
\ \ \Psi ( \sqrt P ) u_j = u_j 
\]
satisfy the following local energy decay estimates:
\begin{equation}
\label{eq:t.2}
\int_{ V } \left( | u ( t , x )|^2 + |\partial_t u ( t, x ) |^2 
\right) 
dx  \ \leq \
\begin{cases} \ \ C \exp ( - \alpha t ) \,, & \text{in case $1$,} \\
\ \ C t^{-n+1} \log t \,,   & \text{in case $2$,} \\ 
\ \ C_M t^ {-M} \,, \ \forall\, M > 0 \,, &  \text{in case $3$,}
\end{cases}
\end{equation}
where the constant $ C$ ($ C_M$) depends on $ U $ and $ V $ (and $M$) only.
\end{thm}
\begin{proof}
We first note that 
it is enough to obtain the estimates  $ \chi U ( t ) \chi 
\; : \; H^{K} \rightarrow L^2 $ where $ \chi \in \CIc ( X ) $ and
\[  U (t ) \defeq \frac {\sin t \sqrt P } {\sqrt P } \,. \]
To do that we follow the
standard procedure (see \cite{TZ},\cite[\S 4]{BZ1} 
and reference given there) and 
perform a contour deformation in the integral: 
\begin{equation}
\chi U(t)  (  P+i)^{-K/2} 
\Psi( \sqrt P)\chi= \frac i {2\pi} \int_ {-\infty} ^{ +\infty} e^ {-it 
\lambda } \chi  (  R ( \lambda ) - R ( - \lambda ))  {(\lambda^2 + i)^{-K/2}} 
 \Psi ( \lambda ) \chi d \lambda \,,
\end{equation} 
for $ t > 0 $. The contribution of $ R ( - \lambda ) $ in the
spectral projection can be eliminated by contour deformation when $ t > 0 $
-- see \cite[Sect.4]{TZ}.  Hence
\begin{equation}
\label{eq:1.5} \chi U(t)  (  P+i)^{-K/2} 
\Psi( \sqrt P)\chi= \frac i {2\pi} \int_ {-\infty} ^{ +\infty} e^ {-it 
\lambda } \chi   R ( \lambda )    (\lambda^2 + i)^{-K/2}  
 \Psi ( \lambda ) \chi d \lambda \,, \ \ t > 0 \,.
\end{equation} 

In case 1, i.e., odd dimensions and $P= -\Delta$ 
in the exterior of a (large) ball, we use the estimate \eqref{eq:la2}
to deform the contour to $ \Gamma = \RR - i \gamma $, 
$ 0 < \gamma < \alpha_1 $. This gives \eqref{eq:t.2} in that case.

In the case of a compactly supported perturbation of 
$ - \Delta $ and $n$ even, we have to modify 
this argument because the resolvent has a branching point at $\lambda =0$. 
Thus we deform the contour near $0$ to 
\[
\{ \lambda = x- i c_1 x, x\geq 0\}\cup\{z= x+ ic_1  x, x\leq 0\},\ \text{for 
$c_1 > 0 $, small.}
\]
We use the usual estimate for the resolvent near $0$:
\[
\|\chi R (\lambda)\chi\| \leq C_M |\lambda|^{n-2}|\log \lambda |
\] 
in any sector $ | \arg \lambda | < M $ -- see for instance 
\cite[\S 3]{Zw}. The dominant part of the integral \eqref{eq:1.5}
comes from the contour near $ 0 $ which gives
\[  \int_0^1 x^{n-2} \log x \,  e^{-x t} dx  \leq C t^{-n+1} \log t \,,
\]
which is the estimate in case 2. 

For case 3, that is the case of $ \Psi \not \equiv 1 $, we consider 
the analytic extension of that function, $\widetilde{ \Psi}$, with 
the property that 
$\bar\partial \widetilde{\Psi} = {\mathcal O} ( |\Im z|^\infty ) $ 
(the defining property of the almost analytic extension -- see
\cite[Chapter 8]{DiSj}) is supported in a set where $P$ 
has no resonances -- see Fig.\ref{f:bz}.
We deform \eqref{eq:1.5} to a contour which for $|z|>1$ is the same as 
before, and for $|z|<1$ is as in Fig.\ref{f:bz}.
\begin{figure}[ht]
\begin{center}
\ecriture{\includegraphics[width=10cm]{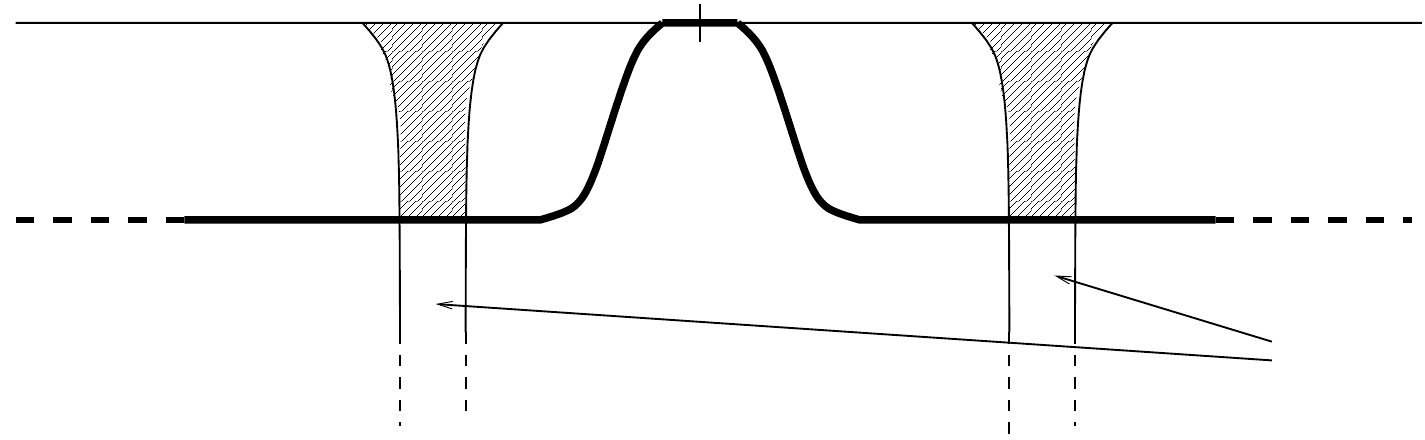}}{ \aat{24}{16}{$0$} \aat{45}{3}{\smash{support of}}\aat{45}{0}{\smash{\quad \,$\bar \partial 
\widetilde{\Psi}$}}\aat{1}{9}{$\widetilde{\Psi}=1$}\aat{22}{6}{$\widetilde{\Psi}=0$}\aat{44}{9}{$\widetilde{\Psi}=1$}}
\end{center}
 \caption{The contour deformation in case 3 and the support properties
of the almost analytic extension of $ \Psi$.\label{fig1}}
\label{f:bz}
\end{figure} 
By Stokes's formula we get exactly the same contributions as in case~1 
(since near $0$, $\widetilde \Psi =0$) with an additional term
\begin{equation}
\frac i {2\pi}\int_{\Omega} \bar \partial \widetilde \Psi \left(z\right) 
  e^ {- it z} 
\chi {(R(z)}{(z^2 + i)^{-L}} \chi \, dz
\end{equation}
where $\Omega$ is the support of $\bar \pa \widetilde \Psi$ between the real axis and the new contour
(shaded in Fig.~\ref{fig1}). 
Since $\bar \partial \widetilde \Psi \left(z\right)= 
{\mathcal {O}}\left( |\Im z|^{\infty}\right)$, 
a repeated integration by parts shows that this last term is 
${\mathcal {O}}\left( t^{-\infty}\right)$ (in the energy norm).
\end{proof}

\medskip
\noindent
{\em Proof of Corollary \ref{cor3}:} We follow the argument of
Burq \cite{Bu}. The left hand side of \eqref{eq:t.2} is 
bounded by the same quantity at $ t = 0 $, and in particular by 
$ \| u _0\|_{H^1}^2 + \| u _1\|_{L^2}^2 $. The estimate
\eqref{eq:t.2} shows that, in case 1 (that is, the case considered
in Corollary \ref{cor3}), it is also bounded by 
$ e^{-\alpha t } \| u_0 \|_{H^{K+1}} + \|u_1\|_{H^{K}} $. Interpolation 
between these two estimates gives \eqref{eq:cor3}. 
\stopthm

\pagebreak
\begin{center}
\textbf{Erratum to 
``Resolvent estimates for normally hyperbolic trapped sets'', Ann. Inst. Henri Poincar\'e (A), {\bf 12}(7)(2011), 1349-1385.}
\end{center}

\setcounter{equation}{0}
\setcounter{figure}{0}
\setcounter{table}{0}
\setcounter{page}{1}
\makeatletter
\renewcommand{\theequation}{S\arabic{equation}}
\renewcommand{\thefigure}{S\arabic{figure}}

In this erratum we correct three errors in the recent paper ``Resolvent estimates for normally hyperbolic trapped sets'', Ann. Inst. Henri Poincar\'e (A), {\bf 12}(7)(2011), 1349-1385.
The errors are minor and do not affect the
correctness of the principal results (although one mild hypothesis
needs to be explicitly added).  Descriptions of these errors and the necessary
corrections are as follows.  Note that this is the second revision of
this erratum, now reflecting the addition of an explicit hypothesis that the
trapped set should be symplectic.

\begin{itemize}

\item In \S 3.5 we omitted a crucial condition on $G$
which is needed to have (3.24). In (3.20) we need 
to strengthen the second condition to 
\[ G = G_1 + \log ( 1/h ) G_2 , \ \ 
   \partial^\alpha H_p^k G_1 = {\mathcal O} ( (h/\tilde h)^{ - |\alpha|/2} ) , \ \ 
k + | \alpha | \geq 1, \ \ 
\partial^\alpha G_2 = {\mathcal O} ( 1 ) . 
\]
This is satisfied for the weight $ G $ in \S\S 4.2--4.3. 
Expression (3.24) holds for $ \ell \geq 2,$ while the
case $ \ell = 1 $ yields the slight variant:
$$
\ad_{G^w(x,hD)} P\in h \log (1/h) \widetilde\Psi_{1/2}.
$$

The analysis follows from \cite[\S 8.2]{S_e-z} 
and is the same as in \cite[\S 8]{S_SZ10}. See
also \cite[\S 7]{S_Da-Dy} and 
\cite[Proposition 4.2]{S_Z} for similar arguments. 

\item

Lemma 4.1 is incorrect as stated. The conclusion 
(4.4) does \emph{not} hold for any defining functions
of $ \Gamma_\pm $ as can be seen by multiplying
$ \varphi_\pm $ by $ e^f $ and having $ |H_f | $ 
large somewhere. \emph{We are grateful to 
Semyon Dyatlov for pointing this out.}

The error in the proof comes from the fact that
$ C $ in the second displayed formula there may 
be greater than $ 1 $.

The simple correction is to state that there exists \emph{some} choice of
defining functions satisfying (4.4) in some neighbourhood
of $ K$. We start with  given defining functions $ \tilde \varphi_\pm $
and then, similarly as in \cite[Proof of Proposition 7.4]{S_SZ10}
(but for defining functions rather than their squares as in \cite{S_SZ10}), 
set
$$
\hph_\pm ( \rho )
\stackrel{\rm{def}}{=} 
\int_0^T \tilde \varphi_{\pm}(\exp t H_p (\rho)) \, dt.
$$
These are defining functions of $ \Gamma_\pm $ as these sets are
invariant under the flow.

Then 
\[ H_p \hph_\pm ( \rho ) = \tilde \varphi_\pm ( \exp T H_p (\rho ))
- \tilde \varphi_\pm ( \rho ) .\]
Since $ | \tilde \varphi_\pm ( \rho ) | \sim d ( \rho, \Gamma_\pm ) $, 
the second displayed formula in the proof of 
Lemma 4.1 with $ T $ large enough
(for $ \rho $ in a $ T $ dependent neighbourhood of  $K $), 
shows that
\[    | \tilde \varphi_+ ( \exp T H_p ( \rho )) | \ll | \tilde 
\varphi_+ ( \rho ) | 
, \ \ | \tilde \varphi_- ( \rho ) | \ll | \tilde 
\varphi_- ( \exp T H_p ( \rho ))|.
\]
Hence 
\[ H_p \hph_+  ( \rho ) \sim - \tilde \varphi_+ ( \rho ) \sim 
- \hph_+ ( \rho) , \ \ 
  H_p \hph_-  ( \rho ) \sim  \tilde \varphi_- ( \exp TH_p ( \rho) ) \sim 
\hph_- ( \rho) , \]
with constants depending on $ T $.
This gives (4.4).

\item  The assertion, in Dynamical Hypothesis (2), that $K$ must
  automatically be
  symplectic, seems to be false.  \emph{We must therefore add the
    hypothesis that $K$ is symplectic,} as this fact is used crucially
  in the end of the proof of Lemma~4.1, where we observe that $\{\varphi_+,
  \varphi_-\} \neq 0.$  \emph{We are grateful to 
Semyon Dyatlov for pointing this out.}

\end{itemize}

\end{document}